\documentclass[11pt,reqno]{amsart}

\newtheorem{thm}{Theorem}[section]
\newtheorem*{thm*}{Theorem}
\newtheorem{lem}[thm]{Lemma}
\newtheorem*{lem*}{Lemma}

\newtheorem{cor}[thm]{Corollary}
\newtheorem{claim}[thm]{Claim}
\newtheorem{prop}[thm]{Proposition}

\theoremstyle{definition}

\newtheorem{assump}[thm]{Assumption}
\renewcommand{\thecase}{}
\newtheorem*{case*}{Case}

\newtheorem*{defn*}{Definition}

\newtheorem*{exmp*}{Example}

\newtheorem{rmk}[thm]{Remark}
\newtheorem*{rmk*}{Remark}
\newtheorem{step}{Step}\renewcommand{\thestep}{}

\theoremstyle{remark}

\makeatletter
\def\alphenumi{
  \def\theenumi{\alph{enumi}}
  \def\p@enumi{\theenumi}
  \def\labelenumi{(\@alph\c@enumi)}}
\makeatother

\makeatletter
\def\thecase{\@arabic\c@case}
\makeatother

\makeatletter
\def\thestep{\@arabic\c@step}
\makeatother

\newcount\hh
\newcount\mm
\mm=\time
\hh=\time
\divide\hh by 60
\divide\mm by 60
\multiply\mm by 60
\mm=-\mm
\advance\mm by \time
\def\hhmm{\number\hh:\ifnum\mm<10{}0\fi\number\mm}

\let\oldmarginpar\marginpar
\renewcommand\marginpar[1]{\-\oldmarginpar[\raggedleft\footnotesize #1]%
{\raggedright\footnotesize #1}}

\newcommand\EE{\mathbb{E}}

\newcommand\NN{\mathbb{N}}

\newcommand\PP{\mathbb{P}}

\newcommand\RR{\mathbb{R}}

\newcommand\cB{{\mathcal{B}}}

\newcommand\cF{{\mathcal{F}}}

\newcommand\cL{{\mathcal{L}}}

\newcommand\eps{\varepsilon}

\newcommand\diag{\operatorname{diag}}

\newcommand\dist{\operatorname{dist}}

\newcommand\loc{\operatorname{loc}}

\newcommand\supp{\operatorname{supp}}

\numberwithin{equation}{section}

\usepackage{amssymb}
\usepackage{hyperref}
\usepackage{mathrsfs}
\usepackage{slashed}
\usepackage[usenames]{color}
\usepackage{url}
\usepackage{verbatim}

\hypersetup{pdftitle={Kimura diffusions}}

\hypersetup{pdfauthor={Camelia A. Pop}}

\begin{document}

\title[Kimura diffusions]{Existence, uniqueness and the strong Markov property of solutions to Kimura diffusions with singular drift}

\author[C. Pop]{Camelia A. Pop}
\address[CP]{Department of Mathematics, University of Pennsylvania, 209 South 33rd Street, Philadelphia, PA 19104-6395}
\email{cpop@math.upenn.edu}

\date{\today{ }\hhmm}

\begin{abstract}
Motivated by applications to proving regularity of solutions to degenerate parabolic equations arising in population genetics, we study existence, uniqueness and the strong Markov property of weak solutions to a class of degenerate stochastic differential equations. The stochastic differential equations considered in our article admit solutions supported in the set 
$[0,\infty)^n\times\RR^m$, and they are degenerate in the sense that the diffusion matrix is not strictly elliptic, as the smallest eigenvalue converges to zero proportional to the distance to the boundary of the domain, and the drift coefficients are allowed to have power-type singularities in a neighborhood of the boundary of the domain. Under suitable regularity assumptions on the coefficients, we establish existence of weak solutions that satisfy the strong Markov property, and uniqueness in law in the class of Markov processes.
\end{abstract}

%



\maketitle

\tableofcontents

\section{Introduction}
\label{sec:Introduction}
The stochastic differential equations considered in our article are a generalization of continuous limits of Markov chains that arise in population genetics as random models for the evolution of gene frequencies. The solutions to such differential equations are supported in $\bar S_{n,m}$, where $S_{n,m}:=\RR_+^n\times\RR^m$, $\RR_+:=(0,\infty)$, and $n$ and $m$ are nonnegative integers such that $n+m\geq 1$. Under suitable regularity assumptions on the coefficients of the stochastic differential equation, we prove existence of weak solutions that satisfy the strong Markov property (Theorem \ref{thm:Existence_SDE_singular}), and we establish that uniqueness in law holds in the class of Markov processes (Theorem \ref{thm:Uniqueness_SDE_singular}). The stochastic differential equations considered in our article take the form:
\begin{equation}
\label{eq:Kimura_SDE_singular}
\begin{aligned}
d X_i(t) &= \left(b_i( Z(t))+ \sqrt{X_i(t)}\sum_{j=1}^{n}f_{ij}( Z(t))h_{ij}(X_j(t))\right)\, dt\\
&\quad+\sqrt{ X_i(t)}\sum_{k=1}^{n+m} \sigma_{ik}( Z(t))\, d W_k(t),\\
d Y_l(t) &= \left(e_l( Z(t))+\sum_{j=1}^{n}f_{n+l,j}( Z(t)) h_{n+l,j}(X_j(t))\right)\, dt+\sum_{k=1}^{n+m} \sigma_{n+l,k}( Z(t))\, d W_k(t),
\end{aligned}
\end{equation}
where $i=1,\ldots,n$ and $l=1,\ldots,m$. The important features of the coefficients of the stochastic differential equation \eqref{eq:Kimura_SDE_singular} are that the diffusion matrix is not strictly elliptic on $S_{n,m}$, in that the smallest eigenvalue converges to $0$ proportional to the distance to the boundary of the domain $S_{n,m}$, the components $h_{ij}(x_j)$ of the drift coefficient are allowed to have power-type singularities of the form $|x_j|^{-q}$, where the positive constant $q$ is suitably chosen, and the coefficient functions $b(z)$ are assumed to be bounded from below by a positive constant on the boundary of the domain $S_{n,m}$. When the coefficients $f_{ij}\equiv 0$, then we only require that the coefficients $b(z)$ are nonnegative on $\partial S_{n,m}$. While the coefficients $f(z)$ and $h(z)$ are assumed to be Borel measurable, the coefficients $b(z)$, $e(z)$ and $(\sigma\sigma^*)(z)$ are assumed to belong to suitable anisotropic H\"older spaces. A precise statement of the properties of the coefficients of the stochastic differential equation \eqref{eq:Kimura_SDE_singular} is given in Assumption \ref{assump:Existence_SDE_singular}, and the definition of the anisotropic H\"older spaces considered in our work is given in \S \ref{subsec:Holder_spaces}.

The stochastic differential equations \eqref{eq:Kimura_SDE_singular} are an extension of continuous processes that arise as continuous limits of discrete models for gene frequencies \cite{Fisher_1922, Wright_1931, Haldane_1932,  Kimura_1957, Kimura_1964, Shimakura_1981, Ethier_Kurtz, KarlinTaylor2}, and we call them generalized Kimura stochastic differential equations with singular drift. When the coefficients $f_{ij}\equiv 0$, the singular drift disappears, and we call the resulting equations standard Kimura stochastic differential equations. 

\subsection{Outline of the article}
\label{subsec:Outline}
We begin in \S \ref{sec:Kimura_SDE} with the analysis of the standard Kimura stochastic differential equation, \eqref{eq:Kimura_SDE}. Existence of solutions (Proposition \ref{prop:Existence_SDE}) is an immediate consequence of classical results, and for this purpose the assumptions on the coefficients are more general, as outlined in Assumption \ref{assump:Existence_SDE}. We establish uniqueness in law of solutions to the standard Kimura stochastic differential equation in Proposition \ref{prop:Uniqueness_SDE}, under the more restrictive  Assumption \ref{assump:Uniqueness_SDE}. Notice that the drift coefficients $b(z)$ are only assumed to be nonnegative on the boundary of the domain $S_{n,m}$, and that the coefficient functions $b(z)$, $e(z)$ and a suitable combination of the coefficients of the diffusion matrix are assumed to belong to the anisotropic H\"older spaces introduced in \S \ref{subsec:Holder_spaces}. This condition arises because our method of proof is based on the existence, uniqueness and regularity of solutions in anisotropic H\"older spaces to the homogeneous initial-value problem,
\begin{equation}
\label{eq:Initial_value_problem_widehatL}
\begin{aligned}
u_t-\widehat Lu &= 0\quad\hbox{on } (0,T)\times S_{n,m},\\
u(0,\cdot)&=f\quad\hbox{on } S_{n,m},
\end{aligned}
\end{equation}
where the operator $\widehat L$ is the generator of standard Kimura diffusions. Regularity of solutions to parabolic equations defined by the infinitesimal generator of standard Kimura diffusions are established in \cite{Epstein_Mazzeo_2010, Epstein_Mazzeo_annmathstudies, Pop_2013b}. Our definition of the anisotropic H\"older spaces in \S \ref{subsec:Holder_spaces} are an adaptation to our framework of the H\"older spaces introduced in \cite[Chapter 5]{Epstein_Mazzeo_annmathstudies}.

In \S \ref{sec:Kimura_SDE_singular}, we prove our main results (Theorems \ref{thm:Existence_SDE_singular} and \ref{thm:Uniqueness_SDE_singular}) concerning the existence and uniqueness in law of weak solutions to the singular Kimura stochastic differential equation, \eqref{eq:Kimura_SDE_singular}. Our method of the proof consists in applying Girsanov's Theorem \cite[Theorem 3.5.1]{KaratzasShreve1991} to the weak solutions of the standard Kimura stochastic differential equation, \eqref{eq:Kimura_SDE}, to change the probability distribution so that, under the new measure, the solutions solve the singular Kimura stochastic differential equation, \eqref{eq:Kimura_SDE_singular}. We justify the application of Girsanov's Theorem by proving that Novikov's condition \cite[Corollary 3.5.13]{KaratzasShreve1991} holds, a fact that uses the Markov property of the processes we consider. Because Girsanov's Theorem is also used in the proof of uniqueness in law of weak solutions, our uniqueness result is established in the class of Markov processes. While this result is sufficient for the applications we have in mind (see \S \ref{subsec:Applications}), employing ideas used to prove \cite[Theorem 12.2.4]{Stroock_Varadhan}, it may be possible to prove that uniqueness in the class of Markov processes implies weak uniqueness. Notice  though that \cite[Theorem 12.2.4]{Stroock_Varadhan} does not apply directly to our framework because our drift coefficients are not necessarily bounded (see condition \eqref{eq:Singularity_h}). When the drift coefficients are bounded, that is, we consider the standard Kimura stochastic differential equation \eqref{eq:Kimura_SDE}, then we establish the weak uniqueness of solutions (Proposition \ref{prop:Uniqueness_SDE}).

To prove existence and uniqueness of weak solutions to the singular Kimura stochastic differential equation, \eqref{eq:Kimura_SDE_singular}, we assume that the drift coefficient functions $b(z)$ are bounded from below on $\partial S_{n,m}$ by a positive constant, $b_0$ (see condition \eqref{eq:Positivity_cond}). This is a crucial ingredient in our verification of Novikov's condition in Lemmas \ref{lem:Novikov_cond} and \ref{lem:Novikov_cond_singular}. Notice also that the singular coefficients $h_{ij}(x_j)$ are assumed to satisfy the growth assumption \eqref{eq:Singularity_h}, where $q\in (0,q_0)$ and the positive constant $q_0$ depends on $b_0$, by  identity \eqref{eq:Choice_q_0}.

\subsection{Applications of the main results}
\label{subsec:Applications}
The motivation to study the \emph{singular} Kimura stochastic differential equation, \eqref{eq:Kimura_SDE_singular}, comes from its application to the proof of the Harnack inequality for nonnegative solutions to the parabolic equation defined by the infinitesimal generator of \emph{standard} Kimura diffusions, which we establish in joint work with Charles Epstein \cite{Epstein_Pop_2013b}. Let $\widehat L$ be the generator of standard Kimura diffusions. Our method of the proof of the Harnack inequality for nonnegative solution to the parabolic equation $u_t-\widehat Lu=0$ consists in employing a stochastic analysis method due to K.-T. Sturm \cite{Sturm_1994}. This makes use of the fact that we already know that the Harnack inequality holds for nonnegative solutions to a parabolic equation $u_t-Lu=0$, where the operator $L$ is a suitable lower order perturbation of the operator $\widehat L$. In 
\cite[\S 4]{Epstein_Mazzeo_cont_est}, C. Epstein and R. Mazzeo show that this is indeed true, when the operator $L$ is chosen to be the infinitesimal generator of \emph{singular} Kimura diffusions which solve equation \eqref{eq:Kimura_SDE_singular}, where the coefficients $h_{ij}(z)$ have the form
$$
h_{ij}(x_j) = \ln x_j \varphi(x_j),\quad\forall\, i=1,\ldots,n+m,\quad\forall\, j=1,\ldots,n,
$$
where $\varphi:\RR\rightarrow[0,1]$ is a compactly supported smooth function. Notice that the preceding form of the coefficients $h_{ij}(z)$ satisfy our growth assumption \eqref{eq:Singularity_h}, and so, Theorems \ref{thm:Existence_SDE_singular} and \ref{thm:Uniqueness_SDE_singular} give us the existence and uniqueness of strong Markov solutions to the singular Kimura stochastic differential equation with logarithmic drift. This together with the strong Markov property of solutions are one of the main ingredients in our proof of the Harnack inequality for nonnegative solution to the parabolic equations defined by the generators of \emph{standard} Kimura diffusions, \eqref{eq:Kimura_SDE}.

\subsection{Comparison with previous research}
\label{subsec:Comparison_previous_research}
Articles which address the questions of existence and uniqueness in law of weak solutions to degenerate stochastic differential equations similar to ours are \cite{Athreya_Barlow_Bass_Perkins_2002, Bass_Perkins_2003}. While the motivation behind the work in \cite{Athreya_Barlow_Bass_Perkins_2002, Bass_Perkins_2003} are applications to superprocesses (\cite[p. 3]{Athreya_Barlow_Bass_Perkins_2002}, \cite[Example 1.4]{Bass_Perkins_2003}), the main application of our results is to the study of diffusions arising in population genetics (\cite{Kimura_1957, Kimura_1964, Shimakura_1981}, \cite[\S 10.1]{Ethier_Kurtz}, \cite[\S 15.2.F]{KarlinTaylor2}), and to the study of regularity of solutions to degenerate parabolic equations (see \S \ref{subsec:Applications}). The main difference between the Kimura stochastic differential equations \eqref{eq:Kimura_SDE} and those considered in \cite{Athreya_Barlow_Bass_Perkins_2002, Bass_Perkins_2003} consist in the fact that we allow coordinates, $\{Y(t)\}_{t\geq 0}$, of the weak solutions whose dispersion coefficients are non-zero on the boundary of the domain $S_{n,m}$, and we do not require the drift coefficients to be bounded; instead we allow singularities in the drift component of the form $|x_i|^{-q}$, for $i=1,\ldots,n$, where the exponent $q$ satisfies a suitable restriction given by inequality \eqref{eq:Singularity_h}. In the sequel, we explain in more detail the differences between the work done in \cite{Athreya_Barlow_Bass_Perkins_2002, Bass_Perkins_2003} and our results.

In \cite{Athreya_Barlow_Bass_Perkins_2002}, the authors consider diffusions corresponding to the generator
$$
\cL u=\sum_{i=1}^n \left(x_i\gamma_i(x)u_{x_ix_i}+b_i(x)u_{x_i}\right),
$$
where $x\in\RR^n_+$ and $u\in C^2(\RR^n_+)$. Under the assumption that the coefficients of the operator $\cL$ are continuous functions on $\bar \RR^n_+$ and that the drift coefficients are positive on $\partial\RR^n_+$, it is proved in \cite{Athreya_Barlow_Bass_Perkins_2002} that the martingale problem associated to the operator $\cL$ has a unique solution. The method of the proof consists in proving $L^2$-estimates for the resolvent operators, employing a method of Krylov and Safonov to establish continuity of the resolvent operators \cite[\S V.7]{Bass_1998}, and a localizing procedure due to Stroock and Varadhan \cite[Theorem 6.6.1]{Stroock_Varadhan} to reduce the existence and uniqueness of solutions to a local problem. In \S \ref{sec:Kimura_SDE}, we recover and extend the results obtained in \cite{Athreya_Barlow_Bass_Perkins_2002}, under the assumption that the coefficients of the operator $\cL$ belong to the anisotropic H\"older spaces introduced in \S \ref{subsec:Holder_spaces}, and we allow the drift coefficient to be $0$ along the boundary of $\RR^n_+$. Moreover, our method of the proof appears to be simpler, as we rely on existence and uniqueness of solutions in anisotropic H\"older spaces to homogeneous initial-value parabolic equations defined by the operator $\cL$. These results were established in \cite{Epstein_Mazzeo_2010, Epstein_Mazzeo_annmathstudies, Epstein_Mazzeo_cont_est_diag, Pop_2013b}.

In \cite{Bass_Perkins_2003}, the authors consider a more general class of generators which are assumed to take the form
$$
\cL u=\sum_{i,j=1}^n \sqrt{x_ix_j}\gamma_{ij}(x)u_{x_ix_i}+\sum_{i=1}^n b_i(x)u_{x_i},
$$
where $x\in\RR^n_+$ and $u\in C^2(\RR^n_+)$. In this work, the coefficient functions $(\gamma(z))$ are $b(z)$ are assumed to belong to suitable weighted H\"older spaces, as opposed to the anisotropic H\"older spaces introduced in \S \ref{subsec:Holder_spaces}, and the drift coefficient $b(z)$ is assumed nonnegative on the boundary of the domain $\RR^n_+$. Using estimates of the semigroup associated to the operator $\cL$ and of the resolvent operators in weighted H\"older spaces, and the localizing procedure of Stroock and Varadhan \cite[Theorem 6.6.1]{Stroock_Varadhan}, the authors prove existence and uniqueness of solutions to the martingale problem associated to $\cL$. Our results are both more general and more restrictive in certain ways, than the ones obtained in \cite{Bass_Perkins_2003}. The smallness condition \cite[Inequality (1.4)]{Bass_Perkins_2003} on the cross-terms $\gamma_{ij}(z)$, for $i\neq j$, of the operator $\cL$ is less restrictive than our analogous condition \eqref{eq:tilde_a} of the matrix $(a(z))$, defined in \eqref{eq:Matrix_a}. On the other hand, we allow non-generate directions, $\{Y(t)\}_{t\geq 0}$, in our stochastic differential equation \eqref{eq:Kimura_SDE_singular}, and we allow singular, unbounded drift coefficients.

\subsection{Notations and conventions}
\label{subsec:Notations_conventions}
Let $F$ be a closed set in $\RR^{n+m}$, and $k$ be a positive integer. We let $C_{\loc}(F;\RR^k)$ denote the set of functions, $f:F\rightarrow \RR^k$, that are continuous on $F$, but are not necessarily bounded. The space $C^{\infty}_c(\bar S_{n,m})$ consists of smooth functions, $f:\bar S_{n,m}\rightarrow\RR$, with compact support in $\bar S_{n,m}$. For a Borel measurable set $U$, we denote by $\cB(U)$ the collection of Borel measurable subsets of $U$.

\subsection{Acknowledgment}
The author would like to thank Charles Epstein for suggesting this problem and for many very helpful discussions on this subject.

\section{Standard Kimura diffusions}
\label{sec:Kimura_SDE}
To establish existence, uniqueness and the strong Markov property of weak solutions to the Kimura stochastic differential equation with singular drift \eqref{eq:Kimura_SDE_singular}, we first prove these results for the standard Kimura diffusions,
\begin{equation}
\label{eq:Kimura_SDE}
\begin{aligned}
d\widehat X_i(t) &= b_i(\widehat Z(t))\, dt+\sqrt{\widehat X_i(t)}\sum_{k=1}^{n+m} \sigma_{ik}(\widehat Z(t))\, d\widehat W_k(t),\quad\forall\, t>0,\\
d\widehat Y_l(t) &= e_l(\widehat Z(t))\, dt+\sum_{k=1}^{n+m} \sigma_{n+l,k}(\widehat Z(t))\, d\widehat W_k(t),\quad\forall\, t>0,
\end{aligned}
\end{equation}
where $i=1,\ldots,n$, and $l=1,\ldots,m$. We denote by $\widehat Z(t)=(\widehat X(t),\widehat Y(t))$, for all $t\geq 0$, the weak solution to the standard Kimura equation \eqref{eq:Kimura_SDE}. We organize this subsection into 
three parts. In \S \ref{subsec:Existence_SDE}, we prove under suitable hypotheses (Assumption \ref{assump:Existence_SDE}) that the standard Kimura stochastic differential equation \eqref{eq:Kimura_SDE} admits weak solutions, $\{\widehat Z(t)\}_{t\geq 0}$, supported in $\bar S_{n,m}$, when the initial condition is assumed to satisfy $\widehat Z(0)\in\bar S_{n,m}$. In \S \ref{subsec:Uniqueness_SDE}, we prove under more restrictive hypotheses (Assumption \ref{assump:Uniqueness_SDE}), that the weak solutions to the Kimura equation \eqref{eq:Kimura_SDE} are unique in law and satisfy the strong Markov property. In \S \ref{subsec:Holder_spaces}, we introduce the definitions of the anisotropic H\"older spaces used in the proof of uniqueness of weak solutions.

\subsection{Existence of weak solutions}
\label{subsec:Existence_SDE}
Existence of solutions to the standard Kimura stochastic differential equation \eqref{eq:Kimura_SDE} can be established for a more general form of the diffusion matrix than the one implied by equations \eqref{eq:Kimura_SDE}. For this reason, we consider the  stochastic differential equation,
\begin{equation}
\label{eq:Kimura_SDE_general}
\begin{aligned}
d\widehat X_i(t) &= b_i(\widehat Z(t))\, dt+\sum_{k=1}^{n+m} \varsigma_{ik}(\widehat Z(t))\, d\widehat W_k(t), \quad\forall\, t>0,\\
d\widehat Y_l(t) &= e_l(\widehat Z(t))\, dt+\sum_{k=1}^{n+m} \varsigma_{n+l,k}(\widehat Z(t))\, d\widehat W_k(t),\quad\forall\, t>0,
\end{aligned}
\end{equation}
where $i=1,\ldots,n$, and $l=1,\ldots,m$. 

To establish existence of weak solutions to the stochastic differential equation \eqref{eq:Kimura_SDE_general}, we need the following

\begin{assump}[Properties of the coefficients in \eqref{eq:Kimura_SDE_general}]
\label{assump:Existence_SDE}
The coefficient functions of the stochastic differential equation \eqref{eq:Kimura_SDE_general} satisfy the properties:
\begin{enumerate}
\item[1.] We assume that $b\in C_{\loc}(\bar S_{n,m};\RR^n)$, $e\in C_{\loc}(\bar S_{n,m};\RR^m)$, and $\varsigma\in C_{\loc}(\bar S_{n,m};\RR^{(n+m)\times (n+m)})$.
\item[2.] The coefficients $b(z)$, $e(z)$ and $(\varsigma(z))$ have at most linear growth in $|z|$.
\item[3.] We assume that
\begin{equation}
\label{eq:Zero_diffusion}
(\varsigma\varsigma^*)_{ii}(z)=0,\quad\forall z\in \left(\partial S_{n,m}\cap\{x_i=0\}\right),\quad\forall\, i=1,\ldots,n,
\end{equation}
where $\varsigma^*$ denotes the transpose matrix of $\varsigma$.
\item[4.] The drift coefficients satisfy
\begin{equation}
\label{eq:Nonnegative_drift}
b_i(z) \geq 0,\quad\forall z\in \left(\partial S_{n,m}\cap\{x_i=0\}\right),\quad\forall\, i=1,\ldots,n.
\end{equation}
\end{enumerate}
\end{assump}

We begin with

\begin{prop}[Existence of weak solutions to standard Kimura diffusions]
\label{prop:Existence_SDE}
Suppose that Assumption \ref{assump:Existence_SDE} holds. Then, for all $z\in\bar S_{n,m}$, there is a weak solution, $(\widehat Z=(\widehat X,\widehat Y),\widehat W)$, on a filtered probability space satisfying the usual conditions, $(\Omega,\{\cF(t)\}_{t\geq 0},\cF,\widehat\PP^{z})$, to the stochastic differential equation \eqref{eq:Kimura_SDE_general}, with initial condition $\widehat Z(0)=z$. Moreover, the weak solution, $\widehat Z=(\widehat X,\widehat Y)$, is supported in $\bar S_{n,m}$.
\end{prop}

\begin{proof}
The method of the proof is similar to that of \cite[Proposition 3.1 and Theorem 3.3]{Feehan_Pop_mimickingdegen_probability}. We divide the proof into two steps. In Step \ref{step:Extended_coeff}, we continuously extend the coefficients of the stochastic differential equation \eqref{eq:Kimura_SDE_general} from $\bar S_{n,m}$ to $\RR^{n+m}$, and we prove that the stochastic differential equation associated to the extended coefficients, \eqref{eq:SDE_extended_coeff}, has a weak solution. In Step \ref{step:Support_solutions}, we prove that any weak solution to equation \eqref{eq:SDE_extended_coeff} is supported $\bar S_{n,m}$, when the support of the initial condition is contained in $\bar S_{n,m}$. Combining Steps \ref{step:Extended_coeff} and \ref{step:Support_solutions}, we obtain the existence of weak solutions supported in $\bar S_{n,m}$, to the stochastic differential equation \eqref{eq:Kimura_SDE_general}.

\setcounter{step}{0}
\begin{step}[Extension of the coefficients]
\label{step:Extended_coeff}
By Assumption \ref{assump:Existence_SDE}, we can extend the coefficients of the stochastic differential equation \eqref{eq:Kimura_SDE_general} by continuity from $\bar S_{n,m}$ to $\RR^{n+m}$. We consider the function $\varphi:\RR^{n+m}\rightarrow \bar S_{n,m}$ defined by
$$
\varphi(z)=z',\hbox{ such that } z'\in\bar S_{n,m} \hbox{ and } |z-z'|=\dist(z,\bar S_{n,m}).
$$
Because $\bar S_{n,m}$ is a closed, convex set, the point $z'\in\bar S_{n,m}$ is uniquely determined for all $z\in\RR^{n+m}$. Moreover $\varphi$ is a continuous function and $\varphi(z)=z$, for all $z \in \bar S_{n,m}$. We define the coefficient functions $\tilde b:=b\circ\varphi$, $\tilde d:=d\circ\varphi$ and $\tilde \varsigma:=\varsigma\circ\varphi$, which are continuous extensions to $\RR^{n+m}$ of the coefficient functions $b$, $d$ and $\sigma$, respectively. By Assumption \ref{assump:Existence_SDE}, the extended coefficients are continuous functions on $\RR^{n+m}$, and have at most linear growth in the spatial variable. By \cite[Theorem 5.4.22]{KaratzasShreve1991}, \cite[Theorem 5.3.10]{Ethier_Kurtz}, it follows that the stochastic differential equation,
\begin{equation}
\label{eq:SDE_extended_coeff}
\begin{aligned}
d\widetilde X_i(t) &= \tilde b_i(\widetilde Z(t))\, dt+\sum_{k=1}^{n+m} \tilde\varsigma_{ik}(\widetilde Z(t))\, d\widetilde W_k(t),\quad\forall\, t>0,\quad\forall\, i=1,\ldots,n\\
d\widetilde Y_l(t) &= \tilde e_l(\widetilde Z(t))\, dt+\sum_{k=1}^{n+m} \tilde \varsigma_{n+l;k}(\widetilde Z(t))\, d\widetilde W_k(t),\quad\forall\, t>0,\quad\forall\, l=1,\ldots,m,
\end{aligned}
\end{equation}
has a weak solution, $\{\widetilde Z(t)=(\widetilde X(t), \widetilde Y(t)), \widetilde W(t)\}_{t\geq 0}$, on a filtered probability space satisfying the usual conditions, $(\Omega,\{\cF(t)\}_{t\geq 0},\cF,\widetilde\PP)$, for any initial condition, $\widetilde Z(0)$.
\end{step}

\begin{step}[Support of weak solutions]
\label{step:Support_solutions}
Let $z \in \bar S_{n,m}$, and let $\{\widetilde Z(t)=(\widetilde X(t), \widetilde Y(t))\}_{t\geq 0}$ be a weak solution to the stochastic differential equation \eqref{eq:SDE_extended_coeff}, with initial condition $\widetilde Z(0)=z$. Our goal is to show that
\begin{equation}
\label{eq:Law_support}
\widetilde\PP^{z}\left(\widetilde Z(t) \in \bar S_{n,m}\right)=1,\quad\forall\, t \geq 0,
\end{equation}
where $\widetilde\PP^{z}$ denotes the probability distribution of the process $\{\widetilde Z(t)\}_{t\geq 0}$, with initial condition $\widetilde Z(0)=z$. To prove identity \eqref{eq:Law_support}, it is sufficient to show that
\begin{equation}
\label{eq:Law_support_i}
\widetilde\PP^{z}\left(\widetilde X_i(t) \geq 0\right)=1,\quad\forall\, t \geq 0,\quad\forall\, i=1,\ldots,n.
\end{equation}
For $\eps>0$, let $\eta_{\eps}:\RR\rightarrow [0,1]$ be a smooth function such that $\eta_{\eps}(s)=1$ for $s\leq-\eps$, $\eta_{\eps}(s)=0$ for $s \geq 0$, and $\eta'_{\eps} \leq 0$ on $\RR$. We see that identity \eqref{eq:Law_support_i} holds, if we show that for all $\eps>0$, we have that
\begin{equation}
\label{eq:Law_support_i_eps}
\widetilde\PP^{z}\left(\eta_{\eps}\left(\widetilde X_i(t)\right) = 0\right)=1,\quad\forall\, t \geq 0,\quad\forall\, i=1,\ldots,n.
\end{equation}
Applying It\^o's rule \cite[Theorem 3.3.6]{KaratzasShreve1991} to the process $\{\eta_{\eps}(\widetilde X_i(t))\}_{t\geq 0}$, we obtain 
\begin{equation}
\label{eq:Ito_rule_X_i}
\EE_{\widetilde\PP^{z}}\left[\eta_{\eps}\left(\widetilde X_i(t)\right)\right] = \eta_{\eps}(z) 
+ \EE_{\widetilde\PP^{z}}\left[\int_0^t \tilde b_i(\widetilde Z(s))\eta'_{\eps}\left(\widetilde X_i(s)\right) + \frac{1}{2} (\tilde\varsigma\tilde\varsigma^*)_{ii}(\widetilde Z(s)) \eta''_{\eps}\left(\widetilde X_i(s)\right) \, ds\right].
\end{equation}
From condition \eqref{eq:Nonnegative_drift} and the construction of the extended coefficient $\tilde b_i$, it follows that the drift coefficient $\tilde b_i(z)$ is nonnegative on the support of the function $\eta'_{\eps}$. Using the fact that $\eta'_{\eps} \leq 0$, we obtain
$$
\tilde b_i(\widetilde Z(s))\eta'_{\eps}\left(\widetilde X_i(s)\right) \leq 0,\quad\forall\, s \in [0,t].
$$
From condition \eqref{eq:Zero_diffusion} and the construction of the extended matrix coefficient $\tilde \varsigma$, it follows that $(\tilde\varsigma\tilde\varsigma^*)_{ii}=0$ on the support of $\eta''_{\eps}$. Thus we have 
$$
(\tilde\varsigma\tilde\varsigma^*)_{ii}(\widetilde Z(s)) \eta''_{\eps}\left(\widetilde X_i(s)\right)=0, \quad\forall\, s \in [0,t].
$$
Using now the fact that $\eta_{\eps}(z)=0$, since we choose $z\in \bar S_{n,m}$ and $\eta_{\eps} \equiv 0$ on $\RR_+$, it follows from identity \eqref{eq:Ito_rule_X_i} that
$$
\EE_{\widetilde\PP^{z}}\left[\eta_{\eps}\left(\widetilde X_i(t)\right)\right] \leq 0,
$$
and, because $\eta_{\eps}$ is a nonnegative function, the preceding expression holds with equality. Since $\eps>0$ was arbitrarily chosen, the preceding identity implies \eqref{eq:Law_support_i}, for all $i=1,\ldots,n$, and so, we conclude that \eqref{eq:Law_support} holds. 
\end{step}
Identity \eqref{eq:Law_support} proves that, when started at points in $\bar S_{n,m}$, the weak solutions to the stochastic differential equation \eqref{eq:SDE_extended_coeff} remain in $\bar S_{n,m}$. Because the coefficients of the stochastic differential equations \eqref{eq:Kimura_SDE} and \eqref{eq:SDE_extended_coeff} agree on $\bar S_{n,m}$, we obtain that the weak solutions to \eqref{eq:SDE_extended_coeff} also solve equation \eqref{eq:Kimura_SDE}. This completes the proof of Proposition \ref{prop:Existence_SDE}.
\end{proof}

\begin{rmk}[Existence of weak solutions to the standard Kimura equation]
\label{rmk:Existence_Kimura_SDE}
We now consider a matrix coefficient, $\sigma :\bar S_{n,m}\rightarrow\RR^{n+m}\times\RR^{n+m}$, such that it satisfies the property that by setting 
\begin{equation}
\label{eq:varsigma}
\begin{aligned}
\varsigma_{ij}(z)&:=\sqrt{x_i}\sigma_{ij}(z),\quad\forall\, i=1,\ldots,n,\quad\forall\, j=1,\ldots,n+m,\\
\varsigma_{ij}(z)&:=\sigma_{ij}(z),\quad\forall\, i=n+l,\ldots,m, \quad\forall\, j=1,\ldots,n+m,
\end{aligned}
\end{equation}
the matrix $(\varsigma(z))$ verifies Assumption \ref{assump:Existence_SDE}. Then Proposition \ref{prop:Existence_SDE} implies that there is a weak solution, $\{\widehat Z(t)\}_{t\geq 0}$, to the standard Kimura stochastic differential equation \eqref{eq:Kimura_SDE}, for any initial condition $\widehat Z(0)$ supported in $\bar S_{n,m}$, and that the solution remains supported in $\bar S_{n,m}$ at all subsequence times.
\end{rmk}

\subsection{Anisotropic H\"older spaces}
\label{subsec:Holder_spaces}
Before we can state the assumptions on the coefficients of the Kimura stochastic differential equation \eqref{eq:Kimura_SDE} that will guarantee the uniqueness in law of weak solutions, we first need to introduce a class of anisotropic H\"older spaces adapted to the degeneracy of the diffusion matrix. The following H\"older spaces are a slight modification of the H\"older spaces introduced by C. Epstein and R. Mazzeo in their study of the existence, uniqueness and regularity of solutions to the parabolic problem defined by Kimura operators \cite{Epstein_Mazzeo_2010, Epstein_Mazzeo_annmathstudies}. 

Following \cite[Chapter 5]{Epstein_Mazzeo_annmathstudies}, we need to first introduce a \emph{distance function}, $\rho$, which takes into account the degeneracy of the diffusion matrix of stochastic differential equation \eqref{eq:Kimura_SDE}. We let
\begin{equation}
\label{eq:rho}
\rho((t^0,z^0),(t,z)) := \rho_0(z^0,z) + \sqrt{|t^0-t|},\quad\forall\, (t^0,z^0), (t,z) \in [0,\infty)\times\bar S_{n,m},
\end{equation}
where $\rho_0$ is a distance function in the spatial variables. Because our domain $S_{n,m}$ is unbounded, as opposed to the compact manifolds considered in \cite{Epstein_Mazzeo_annmathstudies}, the properties of the distance function $\rho_0(z^0,z)$ depend on whether the points $z^0$ and $z$ are in a neighborhood of the boundary of $S_{n,m}$, or far away from the boundary of $S_{n,m}$. For any set of indices, $I\subseteq\{1,\ldots,n\}$, we let
\begin{align}
\label{eq:M_I}
M_I:=\left\{z=(x,y) \in S_{n,m}: x_i \in (0,1)\hbox{ for all } i \in I,\hbox{ and } x_j \in (1,\infty)\hbox{ for all } j \in I^c\right\},
\end{align}
where we denote $I^c:=\{1,\ldots,n\}\backslash I$. The distance function $\rho_0$ has the property that there is a positive constant, $c=c(n,m)$, such that for all subsets $I, J\subseteq\{1,\ldots,n\}$, and all $z^0\in M_I$ and $z\in M_J$, we have that
\begin{equation}
\label{eq:Equivalent_intrinsic_metric}
\begin{aligned}
&c\left(\max_{i\in I\cap J} \left|\sqrt{x^0_i}-\sqrt{x_i}\right| + \max_{j\in (I\cap J)^c} |x^0_j-x_j| +\max_{l\in\{1,\ldots,m\}}|y^0_l-y_l|\right)\\
&\leq \rho(z^0,z)\\
&\leq c^{-1}\left(\max_{i\in I\cap J} \left|\sqrt{x^0_i}-\sqrt{x_i}\right| + \max_{j\in (I\cap J)^c} |x^0_j-x_j| +\max_{l\in\{1,\ldots,m\}}|y^0_l-y_l|\right),
\end{aligned}
\end{equation}
Let $\alpha\in (0,1)$. Following \cite[\S 5.2.4]{Epstein_Mazzeo_annmathstudies}, we let $C^{0,\alpha}_{WF}([0,T]\times \bar S_{n,m})$ be the H\"older space consisting of continuous functions, $u:[0,T]\times\bar S_{n,m}\rightarrow \RR$, such that the following norm is finite
$$
\|u\|_{C^{0,\alpha}_{WF}([0,T]\times \bar S_{n,m})} := \sup_{(t,z)\in [0,T]\times\bar S_{n,m})}|u(t,z)|
+ \sup_{\stackrel{(t^0,z^0), (t,z)\in [0,T]\times \bar S_{n,m}}{(t^0,z^0)\neq (t,z)}} 
\frac{|u(t^0,z^0)-u(t,z)|}{\rho^{\alpha}((t^0,z^0), (t,z))}. 
$$
Let $k$ be a positive integer, and $U$ be a set in $S_{n,m}$. We let $C^{k,\alpha}_{WF}([0,T]\times \bar U)$ denote the H\"older space containing functions, $u\in C^k([0,T]\times\bar U)$, such that the derivatives $D^{\tau}_t D^{\zeta}_z$ belong to the space $C^{0,\alpha}_{WF}([0,T]\times \bar U)$, for all $\tau\in\NN$ and $\zeta\in\NN^{n+m}$, such that $2\tau+|\zeta| \leq k$. We endow the space $C^{k,\alpha}_{WF}([0,T]\times \bar U)$ with the norm,
\begin{align*}
\|u\|_{C^{k,\alpha}_{WF}([0,T]\times \bar U)} &:= \sum_{\stackrel{\tau\in\NN, \zeta\in\NN^{n+m}}{2\tau+|\zeta|\leq k}}
 \|D^{\tau}_t D^{\zeta}_z\|_{C^{0,\alpha}_{WF}([0,T]\times \bar U)}.
\end{align*}
We fix a set of indices, $I\subseteq \{1,\ldots, n\}$. Let $U$ be a set such that $U \subseteq M_{I}$. We let $C^{2+\alpha}_{WF}([0,T]\times \bar U)$ denote the H\"older space of functions, $u\in C^{1,\alpha}_{WF}([0,T]\times\bar U)\cap C^2([0,T]\times U)$, such that 
$$
u_t \in C^{0,\alpha}_{WF}([0,T]\times \bar U),
$$
and such that the functions,
\begin{align*}
\sqrt{x_ix_j}u_{x_ix_j}, \sqrt{x_i}u_{x_iy_l}, u_{y_ly_k} &\in C^{0,\alpha}_{WF}([0,T]\times \bar U),
\quad \forall\, i,j\in I,\quad \forall\, l,k=1,\ldots,m,\\
\sqrt{x_i}u_{x_i x_j}, u_{x_jx_k} &\in C^{0,\alpha}_{WF}([0,T]\times \bar U),\quad\forall\, i \in I, \quad\forall\, j,k \in I^c.
\end{align*}
We endowed the space $C^{2+\alpha}_{WF}([0,T]\times \bar U)$ with the norm,
\begin{align*}
\|u\|_{C^{2+\alpha}_{WF}([0,T]\times \bar U)} &:= \|u\|_{C^{1,\alpha}_{WF}([0,T]\times \bar U)} 
+ \sum_{i,j \in I}\|\sqrt{x_ix_j}u_{x_ix_j}\|_{C^{0,\alpha}_{WF}([0,T]\times \bar U)}
+ \sum_{l,k=1}^m\|u_{y_ly_k}\|_{C^{0,\alpha}_{WF}([0,T]\times \bar U)}\\
&\quad+ \sum_{i \in I} \sum_{j \in I^c}\|\sqrt{x_i}u_{x_ix_j}\|_{C^{0,\alpha}_{WF}([0,T]\times \bar U)}
+ \sum_{i \in I} \sum_{l=1}^m\|\sqrt{x_i}u_{x_iy_l}\|_{C^{0,\alpha}_{WF}([0,T]\times \bar U)}\\
&\quad+ \sum_{i, j \in I^c} \|u_{x_ix_j}\|_{C^{0,\alpha}_{WF}([0,T]\times \bar U)}
+ \sum_{i \in I^c} \sum_{l=1}^m \|u_{x_iy_l}\|_{C^{0,\alpha}_{WF}([0,T]\times \bar U)}
+\|u_t\|_{C^{0,\alpha}_{WF}([0,T]\times \bar U)} .
\end{align*}
We now consider the case when $U$ is an arbitrary set in $S_{n,m}$. Then we let $C^{2+\alpha}_{WF}([0,T]\times \bar U)$ denote the H\"older space consisting of functions $u\in C^2([0,T]\times U)$, satisfying the property that 
$$
u\upharpoonright_{\bar U\cap \bar M_{I}}\in C^{2+\alpha}_{WF}([0,T]\times(\bar U \cap \bar M_I)),\quad\forall\, I\subseteq\{1,\ldots,n\}.
$$
We endow the H\"older space $C^{0,2+\alpha}_{WF}([0,T]\times \bar U)$ with the norm
$$
\|u\|_{C^{2+\alpha}_{WF}([0,T]\times \bar U)} = \sum_{I \subseteq \{1,\ldots,n\}} \|u\|_{C^{2+\alpha}_{WF}([0,T]\times(\bar U \cap \bar M_I))}.
$$
When $k=0$, we write for brevity $C^{\alpha}_{WF}([0,T]\times \bar U)$, instead of $C^{0,\alpha}_{WF}([0,T]\times \bar U)$. The elliptic H\"older spaces $C^{k,\alpha}(\bar U)$ and $C^{2+\alpha}_{WF}(\bar U)$ are defined analogously to their parabolic counterparts, and so, we omit their definitions for brevity.

\subsection{Uniqueness and the strong Markov property}
\label{subsec:Uniqueness_SDE}
Our goal is to prove uniqueness in law and the strong Markov property of weak solutions to the standard Kimura stochastic differential equation \eqref{eq:Kimura_SDE}. By \cite[Proposition 5.4.27]{KaratzasShreve1991}, to prove uniqueness in law of weak solutions to the Kimura stochastic differential equation \eqref{eq:Kimura_SDE}, it is sufficient to establish that for all $z\in \bar S_{n,m}$, any two weak solutions, $\{\widehat Z^i(t)\}_{t\geq 0}$, for $i=1,2$, satisfying the property that $\widehat Z^i(0)=z$, for $i=1,2$, have the same one-dimensional marginal distributions. That is, for all functions $f\in C^{\infty}_c(\bar S_{n,m})$ and $T>0$, we have that
\begin{equation}
\label{eq:Matching_marginals}
\EE_{\widehat\PP^{z}_1}\left[f(\widehat Z^1(T))\right] =  \EE_{\widehat\PP^{z}_2}\left[f(\widehat Z^2(T))\right],
\end{equation}
where $\widehat\PP^{z}_i$ denotes the probability distribution of the process $\{\widehat Z^i(t)\}_{t\geq 0}$, with initial condition $\widehat Z^i(0)=z$, for $i=1,2$. 

Before we give the proof of the uniqueness in law of weak solutions to the Kimura stochastic differential equation in \eqref{eq:Kimura_SDE}, we introduce the differential operator $\widehat L$, which will be the infinitesimal generator of the Markovian solutions to the Kimura stochastic differential equation. We let 
\begin{equation}
\label{eq:Matrix_a}
a(z):=\frac{1}{2}\sigma(z)\sigma^*(z),\quad\forall\, z\in\bar S_{n,m}
\end{equation}
and we define
\begin{equation}
\label{eq:Kimura_operator}
\begin{aligned}
\widehat Lu &= \sum_{i,j=1}^n \sqrt{x_ix_j} a_{ij}(z) u_{x_ix_j} + \sum_{i=1}^n\sum_{l=1}^m \sqrt{x_i} \left(a_{i,n+l}(z)+a_{n+l,i}(z)\right) u_{x_iy_l}\\
&\quad +\sum_{l,k=1}^m a_{n+l,n+k}(z) u_{y_ly_k}
 +\sum_{i=1}^n b_i(z)u_{x_i} +\sum_{l=1}^m e_l(z) u_{y_l}.
\end{aligned}
\end{equation}
for all $z\in S_{n,m}$ and all $u \in C^2(S_{n,m})$. Compare the definition of the Kimura differential operator $\widehat L$ with that of the operator $L$ defined in \cite[Identity (1.1)]{Pop_2013b}. Similarly to \cite[Assumption 3.1]{Pop_2013b},  we will need the following assumptions on the coefficients of the Kimura stochastic differential equation \eqref{eq:Kimura_SDE_singular}.

\begin{assump}[Properties of the coefficients in \eqref{eq:Kimura_SDE}]
\label{assump:Uniqueness_SDE}
The coefficient functions of the stochastic differential equation \eqref{eq:Kimura_SDE_singular} satisfy the properties:
Let $\alpha\in (0,1)$, and assume that
\begin{enumerate}
\item[1.] The coefficient functions $b_i(z)$ satisfy the \emph{nonnegativity} condition \eqref{eq:Nonnegative_drift}, for all $i=1,\ldots,n$.
\item[2.] For all $i,j=1,\ldots,n$ such that $i\neq j$, and all $l=1,\ldots,m$, there are functions, $\alpha_{ii}, \tilde\alpha_{ij}, c_{il}:\bar S_{n,m}\rightarrow\RR$, such that
\begin{equation}
\label{eq:tilde_a}
\begin{aligned}
a_{ij}(z)&=\delta_{ij}\alpha_{ii}(z) +\tilde \alpha_{ij}(z)\sqrt{x_ix_j},\quad\forall\, z=(x,y)\in\bar S_{n,m},\\
a_{i,n+l}(z)=a_{n+l,i}(z)&= \frac{1}{2}c_{il}(z)\sqrt{x_i},\quad\forall\, z=(x,y)\in\bar S_{n,m},
\end{aligned}
\end{equation}
where $\delta_{ij}$ denotes the Kronecker delta symbol.
\item[3.] The \emph{strict ellipticity} condition holds: there is a positive constant, $\lambda$, such that for all sets of indices, $I\subseteq \{1,\ldots,n\}$, for all $z \in \bar M_I$, $\xi\in\RR^n$ and $\eta\in\RR^m$, we have 
\begin{equation}
\label{eq:Uniform_ellipticity}
\begin{aligned}
&\sum_{i\in I} \alpha_{ii}(z)\xi_i^2
 +\sum_{i\in I^c} x_i\alpha_{ii}(z)\xi_i^2
 + \sum_{i,j\in I} \tilde \alpha_{ij}(z)\xi_i\xi_j
 +\sum_{i\in I} \sum_{j\in I^c} x_j(\tilde \alpha_{ij}(z)+\tilde \alpha_{ji}(z))\xi_i\xi_j\\
&+ \sum_{i,j\in I^c} x_ix_j\tilde \alpha_{ij}(z)\xi_i\xi_j
+\sum_{i\in I} \sum_{l=1}^m c_{il}(z)\xi_i\eta_l + \sum_{i\notin I} \sum_{l=1}^m x_ic_{il}(z)\xi_i\eta_l
+\sum_{k,l=1}^m a_{n+k,n+l}(z)\eta_k\eta_l\\
&\quad\geq \lambda\left(|\xi|^2+|\eta|^2\right).
\end{aligned}
\end{equation}
\item[4.] The coefficient functions are \emph{H\"older continuous}: for all sets of indices, $I\subseteq \{1,\ldots,n\}$, and for all $i,i'\in I$, $j,j'\in I^c$ and $l,k=1,\ldots,m$, we have that 
\begin{equation}
\label{eq:Holder_cont}
\begin{aligned}
\alpha_{ii},\,  x_j\alpha_{jj},\, \tilde \alpha_{ii'},\, x_j\tilde \alpha_{ij},\, x_j\tilde \alpha_{ji},\, x_jx_{j'}\tilde \alpha_{jj'} &\in C^{\alpha}_{WF}(\bar M_I),\\
a_{n+k,n+l},\, b_i,\, b_j,\, c_{il},\, x_jc_{jl},\, e_l&\in C^{\alpha}_{WF}(\bar M_I).
\end{aligned}
\end{equation}
\end{enumerate}
\end{assump}

Assumption \ref{assump:Uniqueness_SDE} yield some immediate boundedness conditions on the coefficients of the Kimura stochastic differential equation \eqref{eq:Kimura_SDE}, which we now explain.

\begin{rmk}[Boundedness of the coefficient functions $(b(z))$ and $(\sigma(z))$]
\label{rmk:Boundedness_coeff_SDE}
Assumption \ref{assump:Uniqueness_SDE} implies that there is a positive constant, $K$, such that for all $i=1,\ldots,n$ and all $j,l=1,\ldots,n+m$, we have that
\begin{equation}
\label{eq:Boundedness_coeff_SDE}
|b_i(z)|+|\sigma_{jl}(z)| \leq K,\quad\forall\, z\in\bar S_{n,m}.
\end{equation}
The boundedness of the coefficients $(b(z))$ is obvious from \eqref{eq:Holder_cont} and the definition of the anisotropic H\"older spaces in \S \ref{subsec:Holder_spaces}. The boundedness of the matrix coefficient $(\sigma(z))$ follows from identity \eqref{eq:Matrix_a}, and the fact that the matrix $a(z)$ is bounded, as it is implied by identities \eqref{eq:tilde_a} and condition \eqref{eq:Holder_cont}.
\end{rmk}

\begin{rmk}[Boundedness of the matrix coefficient $(\varsigma(z))$]
\label{rmk:Boundedness_coeff_varsigma}
Assumption \ref{assump:Uniqueness_SDE} yields that the matrix coefficient $(\varsigma(z))$, defined in \eqref{eq:varsigma}, is bounded. To see this, let 
\begin{equation}
\label{eq:Diffusion_matrix}
D(z) := \varsigma(z)\varsigma^*(z),\quad\forall\, z\in \bar S_{n,m},
\end{equation}
be the diffusion matrix of the Kimura stochastic differential equation \eqref{eq:Kimura_SDE}. Using \eqref{eq:varsigma}, \eqref{eq:Matrix_a} and \eqref{eq:tilde_a}, it follows that, for all $i,j=1,\ldots,n$ and all $k,l=1,\ldots,m$, we have that 
\begin{equation}
\label{eq:Matrix_D}
\begin{aligned}
D_{ij}(z) &= 2\left(\delta_{ij} x_i\alpha_{ii}(z)+x_ix_j \tilde\alpha_{ij}(z)\right),\\
D_{i,n+k}(z)=D_{n+k,i}(z) &= x_i c_{ik}(z),\\
D_{n+k,n+l}(z) &=  2a_{n+k,n+l}(z).
\end{aligned}
\end{equation}
Using the boundedness of the coefficients implied by condition \eqref{eq:Holder_cont}, it follows that the coefficient matrix $(D(z))$is bounded, and so, identity \eqref{eq:Diffusion_matrix} implies that the coefficient matrix $(\varsigma(z))$ is also bounded.
\end{rmk}

\begin{rmk}[Structure of the operator $\widehat L$]
\label{rmk:Structure_widehat_L}
Condition \eqref{eq:tilde_a} implies that the differential operator $\widehat L$ takes the form:
\begin{align*}
\widehat Lu&= \sum_{i=1}^n \left(x_i\alpha_{ii}(z)u_{x_ix_i} + b_i(z)u_{x_i}\right) + \sum_{i,j=1}^n x_ix_j \tilde \alpha_{ij}(z)u_{x_ix_j} \\
&\quad+\sum_{i=1}^n\sum_{l=1}^m x_i c_{il}(z)u_{x_iy_l} +\sum_{k,l=1}^m a_{n+k,n+l}(z)u_{y_ky_l}+ \sum_{l=1}^m e_l(z)u_{y_l},
\end{align*}
that is, it has the same structure as the operator $L$ defined in \cite[Identity (1.1)]{Pop_2013b}. Comparing Assumption \ref{assump:Uniqueness_SDE} satisfied by the coefficients of the operator $\widehat L$, with \cite[Assumption 3.1]{Pop_2013b} (with the choice $k=0$) satisfied by the operator $L$ in \cite[Identity (1.1)]{Pop_2013b}, we see that the two operators satisfy the same regularity assumptions. Thus, the properties derived in \cite{Pop_2013b} of the operator $L$, also apply to the Kimura differential operator $\widehat L$. In particular, \cite[Theorem 1.4]{Pop_2013b} holds with $L$ replaced by $\widehat L$.
\end{rmk}

We can now state

\begin{prop}[Uniqueness in law of weak solutions to \eqref{eq:Kimura_SDE}]
\label{prop:Uniqueness_SDE}
Suppose that the coefficients of the standard Kimura stochastic differential equation \eqref{eq:Kimura_SDE} satisfy Assumption \ref{assump:Uniqueness_SDE}. Then, for all $z\in\bar S_{n,m}$, there is a unique weak solution, $(\widehat Z=(\widehat X,\widehat Y),\widehat W)$, $(\Omega,\{\cF(t)\}_{t\geq 0}, \cF,\widehat\PP^{z})$, to the stochastic differential equation \eqref{eq:Kimura_SDE}, satisfying the initial condition $\widehat Z(0)=z$.
\end{prop}

\begin{proof}
The method of the proof is similar to that of \cite[Theorem 1.3 and Proposition 3.6]{Feehan_Pop_mimickingdegen_probability}.
As stated at the beginning of \S \ref{subsec:Uniqueness_SDE}, it suffices to show that for all function $f\in C^{\infty}_c(\bar S_{n,m})$, $z\in \bar S_{n,m}$ and $T>0$, if $(\widehat Z^i, \widehat W^i)$, for $i=1,2$, are two weak solutions to the stochastic differential equation \eqref{eq:Kimura_SDE} with initial condition $\widehat Z^i(0)=z$, then identity \eqref{eq:Matching_marginals} holds. Using Remark \ref{rmk:Structure_widehat_L}, we may apply \cite[Theorem 1.4]{Pop_2013b} to the operator $\widehat L$, to conclude that the homogeneous initial-value problem \eqref{eq:Initial_value_problem_widehatL} has a unique solution, $u \in C^{2+\alpha}_{WF}([0,T]\times\bar S_{n,m})$. We now want to apply It\^o's rule \cite[Theorem 3.3.6]{KaratzasShreve1991} to the process $\{u(T-t, \widehat Z^i(t))\}_{t\in [0,T]}$, but we need to be careful because the function $u$ belongs to the H\"older space of functions $C^{2+\alpha}_{WF}([0,T]\times\bar S_{n,m})$, and so, it is not a $C^{1,2}$ function on $[0,T]\times\bar S_{n,m}$. Using the arguments employed to prove \cite[Proposition 3.5]{Feehan_Pop_mimickingdegen_probability}, we can show that It\^o's rule applies to functions $u \in C^{2+\alpha}_{WF}([0,T]\times\bar S_{n,m})$ and solutions to the Kimura stochastic differential equation \eqref{eq:Kimura_SDE}. In particular, the arguments of \cite[Proposition 3.5]{Feehan_Pop_mimickingdegen_probability} immediately give that
\begin{equation}
\label{eq:Ito_with_varsigma}
\begin{aligned}
d u(T-t,\widehat Z^i(t)) &=-(u_t+ \widehat Lu)(T-t,\widehat Z^i(t))\, dt\\
&\quad+ \sum_{j=1}^n\sum_{k=1}^{n+m}\varsigma_{jk}(\widehat Z^i(t))u_{x_j}(T-t,\widehat Z^i(t))\, d\widehat W^i_k(t)\\
&\quad+ \sum_{l=1}^m\sum_{k=1}^{n+m}\varsigma_{n+l,k}(\widehat Z^i(t))u_{y_l}(T-t,\widehat Z^i(t))\, d\widehat W^i_k(t),\quad\forall\, i=1,2,
\end{aligned}
\end{equation}
where we recall the definition of the coefficient matrix $(\varsigma(z))$ in \eqref{eq:varsigma}. The $d\widehat W^i_j(t)$-terms in the preceding identity define martingales because the coefficient matrix $(\varsigma(z))$ is bounded, by Remark \ref{rmk:Boundedness_coeff_varsigma}, and the derivatives $u_{x_j}$ and $u_{y_l}$ are bounded functions on $[0,T]\times\bar S_{n,m}$, by the definition of the anisotropic H\"older space $C^{2+\alpha}_{WF}([0,T]\times\bar S_{n,m})$ in \S \ref{subsec:Holder_spaces}. Combining the preceding observation with the fact that $u$ is a solution to the initial-value problem \eqref{eq:Initial_value_problem_widehatL}, we obtain from identity \eqref{eq:Ito_with_varsigma} that
\begin{equation*}
u(0,z)=\EE_{\PP^{z}_i}\left[f(Z^i(T))\right],\quad\forall z\in\bar S_{n,m},\quad \forall\, i=1,2.
\end{equation*}
In particular, identity \eqref{eq:Matching_marginals} holds, which implies by \cite[Proposition 5.4.27]{KaratzasShreve1991}, that uniqueness in law holds for solutions to the Kimura stochastic differential equation \eqref{eq:Kimura_SDE}. This completes the proof.
\end{proof}

From \cite[Theorem 5.4.20]{KaratzasShreve1991}, we obtain that uniqueness in law of weak solutions to the Kimura stochastic differential equation \eqref{eq:Kimura_SDE} implies that the solutions satisfy the strong Markov property. Thus we have the following corollary to Proposition \ref{prop:Uniqueness_SDE}

\begin{cor}[The strong Markov property]
\label{cor:Strong_Markov}
Suppose that the coefficients of the standard Kimura stochastic differential equation \eqref{eq:Kimura_SDE} satisfy Assumption \ref{assump:Uniqueness_SDE}. For $z\in\bar S_{n,m}$, let $\{\widehat Z(t)\}_{t\geq 0}$ be the unique weak solution to the stochastic differential equation \eqref{eq:Kimura_SDE}, with initial condition $\widehat Z(0)=z$. Then the process $\{\widehat Z(t)\}_{t\geq 0}$ satisfies the strong Markov property.
\end{cor}

\section{Kimura diffusions with singular drift}
\label{sec:Kimura_SDE_singular}
In this section we prove existence, uniqueness in law and the strong Markov property of weak solutions to Kimura stochastic differential equations with \emph{singular drift}, \eqref{eq:Kimura_SDE_singular}. Our strategy of the proof is to apply Girsanov's Theorem \cite[Theorem 3.5.1]{KaratzasShreve1991} to build a new probability measure so that solutions to the standard Kimura stochastic differential equation \eqref{eq:Kimura_SDE} become solutions to the equation with singular drift \eqref{eq:Kimura_SDE_singular}, under the new probability measure. The weak solutions obtained by this method satisfy the strong Markov property. Girsanov's Theorem also allows us to prove that uniqueness in law of weak solutions to equation \eqref{eq:Kimura_SDE_singular} holds, in the class of Markov processes. 

\subsection{Existence of weak solutions}
\label{subsec:Existence_SDE_singular}
To prove existence of weak solutions to the Kimura stochastic differential equation with singular drift \eqref{eq:Kimura_SDE_singular}, we assume that the coefficients satisfy the following conditions. 

\begin{assump}[Properties of the coefficients in \eqref{eq:Kimura_SDE_singular}]
\label{assump:Existence_SDE_singular}
Let $q\in (0,q_0)$, where $q_0$ is given by
\begin{equation}
\label{eq:Choice_q_0}
q_0:=\min\left\{\frac{1}{4},\frac{b_0}{(n+m)K^2}\right\},
\end{equation}
where $K$ is the positive constant appearing in Remark \ref{rmk:Boundedness_coeff_SDE}. We assume that the coefficient of the stochastic differential equation \eqref{eq:Kimura_SDE_singular} satisfy the following properties:
\begin{enumerate}
\item[1.] The functions $(b(z))$, $(e(z))$, and $(\sigma(z))$ satisfy Assumption \ref{assump:Uniqueness_SDE}.
\item[2.] The drift coefficients satisfy the \emph{positivity condition}: there is a positive constant, $b_0$, such that for all $i=1,\ldots,n$ we have that
\begin{equation}
\label{eq:Positivity_cond}
b_i(z) \geq b_0>0,\quad\forall\, z\in \left(\partial S_{n,m}\cap \{x_i=0\}\right).
\end{equation}
\item[3.] The coefficients $f_{ij}:S_{n,m}\rightarrow\RR$ and $h_{ij}:\RR_+\rightarrow\RR$ are Borel measurable functions, for all $i=1,\ldots, n+m$ and all $j=1,\ldots,n$, and there is a positive constant, $K_0$, such that  
\begin{align}
\label{eq:Boundedness_f}
|f(z)| &\leq K_0 \quad\hbox{for a.e. } z \in S_{n,m},\\
\label{eq:Boundedness_sigma_inverse_f}
|\sigma^{-1}(z)f(z)| &\leq K_0 \quad\hbox{for a.e. } z \in S_{n,m},\\
\label{eq:Singularity_h}
|h_{ij}(s)| &\leq K_0 s^{-q}\quad\hbox{for a.e. } s \in \RR_+.
\end{align}
\end{enumerate}
\end{assump}

\begin{rmk}[Invertibility of the matrix coefficient $(\sigma(z))$]
\label{rmk:Invertibility_sigma}
Condition \eqref{eq:Boundedness_sigma_inverse_f} uses the fact that the matrix coefficient $(\sigma(z))$ is invertible on $S_{n,m}$. From identity \eqref{eq:Matrix_a}, the matrix coefficient $(\sigma(z))$ is invertible if and only if $(a(z))$ is invertible. From identities \eqref{eq:tilde_a} and \eqref{eq:Matrix_D}, we see that we can write $D=BaB$, where $B=\left[\diag(\sqrt{x_i}); I_m\right]$. Since $(B(z))$ is invertible on $S_{n,m}$, it remains to show that $(D(z))$ is invertible, in order to conclude that $(a(z))$ is invertible. Note that identity \eqref{eq:Matrix_D} and the strict ellipticity condition \eqref{eq:Uniform_ellipticity} yield that $(D(z))$ is a symmetric, positive definite matrix for all $z\in S_{n,m}$, and so, $(D(z))$ is invertible on $S_{n,m}$. This completes the proof that the matrix coefficient $(\sigma(z))$ is invertible on $S_{n,m}$. 
\end{rmk}

\begin{rmk}[The boundedness condition \eqref{eq:Boundedness_sigma_inverse_f}]
\label{rmk:Boundedness_sigma_inverse_f}
In general, the boundedness condition \eqref{eq:Boundedness_sigma_inverse_f} is not a consequence of Assumption \ref{assump:Uniqueness_SDE}, as it was the case of the invertibility of the matrix coefficient $(\sigma(z))$ on $S_{n,m}$ (see Remark \ref{rmk:Invertibility_sigma}). For the applications we have in mind (see \cite{Epstein_Pop_2013b} and \S \ref{subsec:Applications}), it is sufficient to assume that the coefficient functions $f_{ij}:S_{n,m}\rightarrow\RR$ are \emph{bounded}, Borel measurable functions, for all $i=1,\ldots, n+m$ and all $j=1,\ldots,n$, and that they have support in a small neighborhood of $\{x_i=0,\,\forall\, i=1,\ldots,n\}$. For concreteness, assume that there is $r_0\in (0,1)$ such that
$$
\supp f_{ij} \subseteq \bar A_{r_0} := \left\{x_k\in [0,r_0]:\, k=1,\ldots,n\right\}.
$$
Under such assumptions, condition \eqref{eq:Boundedness_sigma_inverse_f} is a consequence of Assumption \ref{assump:Uniqueness_SDE}. To see this, we show that the matrix coefficient $(\sigma^{-1}(z))$ is bounded on $\bar A_{r_0}$, and using the boundedness of $(f(z))$, we obtain that condition \eqref{eq:Boundedness_sigma_inverse_f} holds. Using identity \eqref{eq:Matrix_a}, we see that the matrix coefficient $(\sigma^{-1}(z))$ is bounded on $\bar A_{r_0}$ if and only if $(a^{-1}(z))$ is bounded on $\bar A_{r_0}$. For this, it is sufficient to prove that $(a(z))$ is a strictly positive definite matrix on $\bar A_{r_0}$. Using identities \eqref{eq:tilde_a}, for all $\xi\in\RR^n$ and $\eta\in\RR^m$, we have that
\begin{align*}
&\sum_{i,j=1}^n a_{ij}(z)\xi_i\xi_j + \sum_{i=1}^n\sum_{l=1}^m \left(a_{i,n+l}(z)+a_{n+l,i}(z)\right)\xi_i\eta_l +\sum_{l,k=1}^m a_{n+l,n+k} \eta_l\eta_k \\
&\quad= \sum_{i=1}^n \alpha_{ii}(z)\xi_i^2 
+ \sum_{i,j=1}^n \sqrt{x_ix_j}\tilde \alpha_{ij}(z)\xi_i\xi_j
+\sum_{i=1}^n \sum_{l=1}^m \sqrt{x_i}c_{il}(z)\xi_i\eta_l 
+\sum_{k,l=1}^m a_{n+k,n+l}(z)\eta_k\eta_l\\
&\quad= \sum_{i=1}^n x_i\alpha_{ii}(z)\xi_i^2 
+ \sum_{i,j=1}^n \sqrt{x_ix_j}\tilde \alpha_{ij}(z)\xi_i\xi_j
+\sum_{i=1}^n \sum_{l=1}^m \sqrt{x_i}c_{il}(z)\xi_i\eta_l 
+\sum_{k,l=1}^m a_{n+k,n+l}(z)\eta_k\eta_l\\
&\quad\quad+\sum_{i=1}^n (1-x_i)\alpha_{ii}(z)\xi_i^2.
\end{align*}
The last inequality and the strict ellipticity condition \eqref{eq:Uniform_ellipticity}, applied with $\sqrt{x_i}\xi$ instead of $\xi_i$, yield
\begin{align*}
&\sum_{i,j=1}^n a_{ij}(z)\xi_i\xi_j + \sum_{i=1}^n\sum_{l=1}^m \left(a_{i,n+l}(z)+a_{n+l,i}(z)\right)\xi_i\eta_l +\sum_{l,k=1}^m a_{n+l,n+k} \eta_l\eta_k \\
&\quad\geq \lambda\left(|x\cdot\xi|^2+|\eta|^2\right)+(1-r_0)\lambda|\xi|^2\\
&\quad \geq (1-r_0)\lambda \left(|\cdot\xi|^2+|\eta|^2\right).
\end{align*}
Thus, indeed the matrix $(a(z))$ is strictly positive definite on $\bar A_{r_0}$, and so, the matrix coefficient $(\sigma^{-1}(z))$ is bounded on $\bar A_{r_0}$.
\end{rmk}

We begin by proving existence of weak solutions to the singular Kimura equation \eqref{eq:Kimura_SDE_singular}. The solutions that we build in Theorem \ref{thm:Existence_SDE_singular} satisfy the strong Markov property.
\begin{thm}[Existence of weak solutions to Kimura equation with singular drift \eqref{eq:Kimura_SDE_singular}]
\label{thm:Existence_SDE_singular}
Suppose that the coefficients of the Kimura stochastic differential equation with singular drift \eqref{eq:Kimura_SDE_singular} satisfy Assumption \ref{assump:Existence_SDE_singular}. Then, for all $z\in\bar S_{n,m}$, there is a weak solution  $\left(Z=(X,Y), W\right)$, $(\Omega,\{\cF_t\}_{t\geq 0},\cF, \PP^{z})$, to equation \eqref{eq:Kimura_SDE_singular}, with initial condition $Z(0)=z$. Moreover, the solution satisfies the strong Markov property. 
\end{thm}

The proof of Theorem \ref{thm:Existence_SDE_singular} is based on an application of Girsanov's Theorem. We change the probability distributions of the weak solutions of the standard Kimura equation \eqref{eq:Kimura_SDE} obtained in Proposition \ref{prop:Existence_SDE}, so that we add a singular drift as in equation \eqref{eq:Kimura_SDE_singular}. In order to justify the application of Girsanov's Theorem, we prove that Novikov's condition, \cite[Corollary 3.5.13]{KaratzasShreve1991}, for standard Kimura diffusions holds.

\begin{lem}[Novikov's condition for standard Kimura diffusions]
\label{lem:Novikov_cond}
Suppose that the coefficients of the Kimura stochastic differential equation \eqref{eq:Kimura_SDE} satisfy Assumption \ref{assump:Uniqueness_SDE} and condition \eqref{eq:Positivity_cond}. Let $q\in (0,q_0)$, where the positive constant $q_0$ is given by \eqref{eq:Choice_q_0}. Then, for all $\Lambda>0$ and $T>0$, we have
\begin{equation}
\label{eq:Novikov_cond}
\sup_{z\in\bar S_{n,m}}\EE_{\widehat \PP^{z}}\left[\hbox{exp}\left(\Lambda\int_0^T  \sum_{i=1}^n |\widehat X_i(t)|^{-2q} \, dt\right)\right]<\infty,
\end{equation}
where $\{\widehat Z(t)=(\widehat X(t), \widehat Y(t))\}_{t\geq 0}$ is the unique weak solution to the Kimura stochastic equation \eqref{eq:Kimura_SDE}, with initial condition $\widehat Z(0)=z$.
\end{lem}

An elementary method to guarantee that condition \eqref{eq:Novikov_cond} holds is to prove that the hypotheses of Khas'minskii's Lemma \cite{Berthier_Gaveau_1978, Khasminskii, Portenko} are satisfied. A statement of Khas'minskii's Lemma when the underlying process is Brownian motion, can be found in \cite[Theorem 1.2]{Aizenman_Simon_1982}. 

\begin{lem}[Verification of the hypotheses of Khas'minskii's Lemma]
\label{lem:Khasminskii}
Suppose that the coefficients of the Kimura stochastic differential equation \eqref{eq:Kimura_SDE} satisfy Assumption \ref{assump:Uniqueness_SDE} and condition \eqref{eq:Positivity_cond}. Let $\Lambda$ be a positive constant and $q\in (0,q_0)$, where the positive constant $q_0$ is given by \eqref{eq:Choice_q_0}. Then for all $\delta\in (0,1)$, there is a positive constant, $T=T(b_0,\delta,K,\Lambda,m,n,q)$, such that
\begin{equation}
\label{eq:Khasminskii}
\sup_{z\in\bar S_{n,m}}\EE_{\widehat \PP^{z}}\left[\int_0^T  \Lambda \sum_{i=1}^n |\widehat X_i(t)|^{-2q} \, dt\right]<\delta,
\end{equation}
where $\{\widehat Z(t)=(\widehat X(t), \widehat Y(t))\}_{t\geq 0}$ is the unique weak solution to the Kimura stochastic equation \eqref{eq:Kimura_SDE}, with initial condition $\widehat Z(0)=z$.
\end{lem}

\begin{proof}
Without loss of generality, we may assume that $\Lambda=1$. Using condition \eqref{eq:Positivity_cond} and the uniform continuity of the coefficient $b_i(z)$ implied by \eqref{eq:Holder_cont}, we obtain that for all $\rho\in (0,1)$, there is a positive constant, $r$, such that
\begin{equation}
\label{eq:Bound_below_drift}
b_i(z) \geq \frac{b_0}{1+\rho}\quad\hbox{on } \{z=(x,y)\in S_{n,m}:\ x_i\in [0,r]\},\quad\forall\, i=1,\ldots,n.
\end{equation}
Let $\varphi:[0,\infty)\rightarrow[0,1]$ be a smooth cut-off function, such that $\varphi(s)=1$ for $s \leq r/2$, and $\varphi(s)=0$ for $s \geq r$, and such that there is a positive constant, $c$, with the property that
\begin{equation}
\label{eq:Derivatives_varphi}
\|\varphi'\|_{C(\RR)} \leq cr^{-1},\quad\hbox{and}\quad \|\varphi''\|_{C(\RR)} \leq c r^{-2}. 
\end{equation}
For all $\eps\in (0,1)$, we let $\widehat X^{\eps}_i(t)=\widehat X_i(t)+\eps$, and $x_i^{\eps}=x_i+\eps$. By It\^o's rule \cite[Theorem 3.3.6]{KaratzasShreve1991} applied to the process $\varphi(X^{\eps}_i(t))(X^{\eps}_i(t))^{1-2q}$, we obtain
\begin{align*}
&d\varphi(\widehat X^{\eps}_i(t))(\widehat X^{\eps}_i(t))^{1-2q} =(1-2q) \varphi(\widehat X^{\eps}_i(t))(\widehat X^{\eps}_i(t))^{-2q} \left(b_i(\widehat Z(t))- q|\sigma_i(\widehat Z(t))|^2\frac{\widehat X_i(t)}{\widehat X^{\eps}_i(t)}\right)\, dt\\
&\qquad+ b_i(\widehat Z(t))\varphi'(\widehat X^{\eps}_i(t))(\widehat X^{\eps}_i(t))^{1-2q}\, dt\\
&\qquad + \frac{|\sigma_i(\widehat Z(t))|^2}{2} \widehat X_i(t) (\widehat X^{\eps}_i(t))^{-2q}\left(\widehat X^{\eps}_i(t)\varphi''(\widehat X^{\eps}_i(t))  +2(1-2q)\varphi'(\widehat X^{\eps}_i(t))\right)\, dt\\
&\qquad+ (\widehat X^{\eps}_i(t))^{-2q}\left(\widehat X^{\eps}_i(t)\varphi'(\widehat X^{\eps}_i(t)) +(1-2q)\varphi(\widehat X^{\eps}_i(t))\right)\sqrt{\widehat X_i(t)}\sigma_i(\widehat Z(t))\cdot\, d\widehat W(t),
\end{align*}
where $\sigma_i(z)$ denotes the $i$-th row of the matrix function $(\sigma(z))$. From Remark \ref{rmk:Boundedness_coeff_varsigma} and identity \eqref{eq:varsigma}, we see that the coefficients $(\sqrt{x_i}\sigma_i(z))$ are bounded, and so, the $d\widehat W(t)$-term in the preceding equality defines a martingale. We obtain
\begin{align*}
&\EE_{\widehat \PP^{z}}\left[\varphi(\widehat X^{\eps}_i(T))(\widehat X^{\eps}_i(T))^{1-2q}\right] 
= \varphi(x^{\eps}_i)(x^{\eps})^{1-2q}\\
&\quad + (1-2q) \EE_{\widehat \PP^{z}}
\left[\int_0^T \varphi(\widehat X^{\eps}_i(t))(\widehat X^{\eps}_i(t))^{-2q} \left(b_i(\widehat Z(t))-q|\sigma_i(\widehat Z(t))|^2\frac{\widehat X_i(t)}{\widehat X^{\eps}_i(t)}\right)\, dt\right]\\
&\quad+ \EE_{\widehat \PP^{z}}\left[\int_0^Tb_i(\widehat Z(t))\varphi'(\widehat X^{\eps}_i(t))(\widehat X^{\eps}_i(t))^{1-2q}\, dt\right]\\
&\quad+ \EE_{\widehat \PP^{z}}\left[\int_0^T\frac{|\sigma_i(\widehat Z(t))|^2}{2} \widehat X_i(t) (\widehat X^{\eps}_i(t))^{-2q}\left(\widehat X^{\eps}_i(t)\varphi''(\widehat X^{\eps}_i(t)) +2(1-2q)\varphi'(\widehat X^{\eps}_i(t))\right)\, dt\right],
\end{align*}
for all $z=(x,y)\in\bar S_{n,m}$. The preceding identity together with the boundedness of the coefficients $(b(z))$ and $(\sigma(z))$ (see inequality \eqref{eq:Boundedness_coeff_SDE}), and the choice of the cut-off function $\varphi$ and \eqref{eq:Derivatives_varphi}, give us that there is a positive constant, $C=C(K,m,n)$, such that
\begin{align*}
(1-2q) \EE_{\widehat \PP^{z}}
\left[\int_0^T \varphi(\widehat X^{\eps}_i(t))(\widehat X^{\eps}_i(t))^{-2q} \left(b_i(\widehat Z(t))- q|\sigma_i(\widehat Z(t))|^2\frac{\widehat X_i(t)}{\widehat X^{\eps}_i(t)}\right)\, dt\right]
\leq r^{1-2q}+C r^{-2q} T.
\end{align*}
Using inequalities \eqref{eq:Bound_below_drift} and \eqref{eq:Boundedness_coeff_SDE}, we see that 
\begin{equation}
\label{eq:Lower_bound}
\begin{aligned}
&(1-2q) \EE_{\widehat\PP^{z}}\left[\int_0^T \varphi(\widehat X^{\eps}_i(t))(\widehat X^{\eps}_i(t))^{-2q} \left(b_i(\widehat Z(t))- q|\sigma_i(\widehat Z(t))|^2\frac{\widehat X_i(t)}{\widehat X^{\eps}_i(t)}\right)\, dt\right]\\
&\quad\geq (1-2q)\left(\frac{b_0}{1+\rho}-q(n+m)K^2\right) \EE_{\widehat \PP^{z}}\left[\int_0^T (\widehat X^{\eps}_i(t))^{-2q}  \mathbf{1}_{\{\widehat X_i(t)\in [0,r/2]\}} \, dt\right].
\end{aligned}
\end{equation}
Combining the preceding two inequalities, and letting $\eps$ tend to $0$, we obtain that there is a positive constant, $C=C(K,m,n)$, such that
\begin{align}
\label{eq:Inequality_X_power}
\EE_{\widehat \PP^{z}}\left[\int_0^T (\widehat X_i(t))^{-2q}  \, dt\right]
\leq \frac{r^{1-2q}+C r^{-2q} T}{(1-2q)\left(\frac{b_0}{1+\rho}-q(n+m)K^2\right)}.
\end{align}
Note that by choosing $q\in (0,q_0)$, where the positive constant $q_0$ is given by \eqref{eq:Choice_q_0}, we can find a positive constant $\rho_0=\rho_0(K,m,n)$, such that
$$
\frac{b_0}{1+\rho}-q(n+m)K^2>0.
$$
Then we choose $r_0=r_0(\|b\|_{C^{\alpha}_{WF}(\bar S_{n,m})},K,m,n)$, such that inequality \eqref{eq:Bound_below_drift} holds with $\rho$ replaced by $\rho_0$, for all $r\in (0,r_0)$. For all $\delta\in (0,1)$, let $r=r(\|b\|_{C^{\alpha}_{WF}(\bar S_{n,m})},\delta,K,m,n,q)$ and $T=T(\delta,K,m,n,q)$ be chosen small enough such that using inequality \eqref{eq:Inequality_X_power}, we obtain that estimate \eqref{eq:Khasminskii} holds. This completes the proof.
\end{proof}

Using Lemma \ref{lem:Khasminskii}, we can now give the proof of

\begin{proof}[Proof of Lemma \ref{lem:Novikov_cond}]
By Corollary \ref{cor:Strong_Markov}, the solutions to standard Kimura stochastic differential equations \eqref{eq:Kimura_SDE} satisfy the Markov property. Thus, the proof of \cite[Theorem 1.2]{Aizenman_Simon_1982} easily adapts to standard Kimura diffusions in place of standard Brownian motion, and using Lemma \ref{lem:Khasminskii}, we obtain that for all $\delta>0$, there is $T_{\delta}>0$, such that 
\begin{equation}
\label{eq:Novikov_cond_fixed_time}
\sup_{z\in\bar S_{n,m}}\EE_{\widehat \PP^{z}}\left[\hbox{exp}\left(\int_0^{T_{\delta}}  \varphi(\widehat X(t))\, dt\right)\right]<\frac{1}{1-\delta},
\end{equation}
where we denote for brevity, $\varphi(x):=\Lambda\sum_{i=1}^n |x_i|^{-2q}$, for all $x\in \RR^n_+$. Let $T>0$ and set $k:=\left\lceil{T/T_{\delta}}\right\rceil$. We consider the sequence $T_i:=T-(k-i)T_{\delta}$, for all $i=1,\ldots,k$, and $T_0=0$. We have, for all $z\in\bar S_{n,m}$,
\begin{align*}
\EE_{\widehat \PP^{z}}\left[\hbox{exp}\left(\int_0^T \varphi(\widehat X(t))\, dt\right)\right]
&= \EE_{\widehat \PP^{z}}\left[ e^{\int_0^{T_{k-1}} \varphi(\widehat X(t))\, dt}
e^{\int_{T_{k-1}}^{T_k} \varphi(\widehat X(t))\, dt} \right]\\
&=\EE_{\widehat \PP^{z}}\left[\EE_{\widehat \PP^{z}}\left[ e^{\int_0^{T_{k-1}} \varphi(\widehat X(t))\, dt}
e^{\int_{T_{k-1}}^{T_k} \varphi(\widehat X(t))\, dt} \Big{|} \cF_{T_{k-1}}\right] \right]\\
&=\EE_{\widehat \PP^{z}}\left[ e^{\int_0^{T_{k-1}} \varphi(\widehat X(t))\, dt}
\EE_{\widehat \PP^{\widehat Z(T_{k-1})}}\left[ 
e^{\int_0^{T_{\delta}} \varphi(\widehat X(t))\, dt} \right] \right].
\end{align*}
Inequality \eqref{eq:Novikov_cond_fixed_time} gives us
\begin{align*}
\EE_{\widehat \PP^{z}}\left[\hbox{exp}\left(\int_0^T \varphi(\widehat X(t))\, dt\right)\right]
&\leq \frac{1}{1-\delta} \EE_{\widehat \PP^{z}}\left[ e^{\int_0^{T_{k-1}} \varphi(\widehat X(t))\, dt}\right],
\end{align*}
and iterating the preceding argument $k$ times, we obtain
\begin{align*}
\EE_{\widehat \PP^{z}}\left[\hbox{exp}\left(\int_0^T \varphi(\widehat X(t))\, dt\right)\right]
&\leq \frac{1}{(1-\delta)^k},\quad\forall z\in\bar S_{n,m}.
\end{align*}
Thus, inequality \eqref{eq:Novikov_cond} now follows.
\end{proof}

Lemma \ref{lem:Novikov_cond} allows us to establish the

\begin{proof}[Proof of Theorem \ref{thm:Existence_SDE_singular}]
We divide the proof into two steps. In Step \ref{step:Existence_SDE_singular}, we prove existence of weak solutions to the Kimura stochastic differential equation with singular drift \eqref{eq:Kimura_SDE_singular}, via Girsanov's Theorem, and in Step \ref{step:Markov_SDE_singular} we show that the solutions constructed in Step \ref{step:Existence_SDE_singular} satisfy the strong Markov property.

\setcounter{step}{0}
\begin{step}[Existence of weak solutions to equation \eqref{eq:Kimura_SDE_singular}]
\label{step:Existence_SDE_singular}
Let $z\in\bar S_{n,m}$. Because the coefficient functions $b$, $e$, and $\sigma$ satisfy Assumptions \ref{assump:Existence_SDE} and \ref{assump:Uniqueness_SDE}, Propositions \ref{prop:Existence_SDE} and \ref{prop:Uniqueness_SDE} show that there is a unique weak solution, $(\widehat Z=(\widehat X,\widehat Y),\widehat W)$, $(\Omega, \{\cF_t\}_{t\geq 0},\cF,\widehat \PP^{z})$, to the Kimura stochastic differential equation \eqref{eq:Kimura_SDE}, with initial condition $\widehat Z(0)=z$. Let $\theta: S_{n,m}\rightarrow\RR^{n+m}$ be the Borel measurable vector field defined by 
\begin{equation}
\label{eq:Definition_theta}
\theta : = \sigma^{-1} \xi,
\end{equation} 
where the function $\xi:S_{n,m}\rightarrow\RR^{n+m}$ is defined by
\begin{equation}
\label{eq:Definition_xi}
\xi_i(z):=\sum_{j=1}^m f_{ij}(z)h_{ij}(x_j)\quad\hbox{for a.e. } z\in S_{n,m},\quad\forall\, i=1,\ldots,n+m.
\end{equation}
From conditions \eqref{eq:Boundedness_sigma_inverse_f} and \eqref{eq:Singularity_h}, it follows that there is a positive constant, $\Lambda$, such that
\begin{equation}
\label{eq:Upper_bound_theta}
|\theta(z)| \leq \Lambda \sum_{i=1}^n |x_i|^{-q} \quad\hbox{for a.e. } z\in S_{n,m}.
\end{equation}
Let $T$ be a positive constant. Lemma \ref{lem:Novikov_cond} together with the preceding inequality shows that condition \eqref{eq:Novikov_cond} holds, and so, \cite[Corollary 3.5.13]{KaratzasShreve1991} implies that the process $\{\widehat M(t)\}_{0\leq t\leq T}$ defined by
\begin{equation*}
\widehat M(t):=\hbox{exp}\left(\int_0^t \theta(\widehat Z(s))\cdot\, d\widehat W(s)- \frac{1}{2}\int_0^t |\theta(\widehat Z(s))|^2\, ds\right),\quad \forall t \in [0,T],
\end{equation*}
is a $\widehat \PP^{z}$-martingale. We can apply Girsanov's Theorem \cite[Theorem 3.5.1]{KaratzasShreve1991} to construct a new probability measure, $\PP^{z}$, by letting
\begin{equation}
\label{eq:Definition_probability_measure_SDE_singular}
\frac{d\PP^{z}}{d\widehat \PP^{z}} = \widehat M(T),
\end{equation}
such that the process $W(t):=\widehat W(t)-\int_0^t \theta(\widehat Z(s))\, ds$, for all $t\in [0,T]$, is a Brownian motion with respect to the probability measure $\PP^{z}$. Using \eqref{eq:Definition_theta}, we see that by letting $Z(t):=\widehat  Z(t)$, for all $t\in [0,T]$, we obtain that the process $\{Z(t), W(t)\}_{0\leq t\leq T}$, $(\Omega,\{\cF_t\}_{0\leq t\leq T}, \cF, \PP^{z})$ is a weak solution to the Kimura stochastic differential equation with singular drift \eqref{eq:Kimura_SDE_singular}, with initial condition $Z(0)=z$.
\end{step}

\begin{step}[The strong Markov property]
\label{step:Markov_SDE_singular}
Let $z\in \bar S_{n,m}$, and let $(Z,W)$, $(\Omega,\{\cF_t\}_{t\geq 0},\cF,\PP^z)$ be the weak solution to the Kimura stochastic differential equation with singular drift \eqref{eq:Kimura_SDE_singular}, with initial condition $Z(0)=z$, constructed in Step \ref{step:Existence_SDE_singular}. We now show that the process $\{Z(t)\}_{t\geq 0}$ satisfies the strong Markov property, that is, for all stopping times, $\tau$, for all $t \geq 0$, and $B\in \cB(\bar S_{n,m})$, we have that
\begin{equation}
\label{eq:Strong_Markov_singular}
\PP^z\left(Z(\tau+t) \in B|\cF_{\tau}\right) = \PP^z\left(Z(\tau+t) \in B|Z(\tau)\right)\quad \PP^z\hbox{-a.s. on } \{\tau<\infty\}.
\end{equation}
It is sufficient to prove identity \eqref{eq:Strong_Markov_singular} for all bounded stopping times in order to conclude that the strong Markov property \eqref{eq:Strong_Markov_singular} holds for arbitrary stopping times. Let $T>0$ and let $\tau$ be a stopping time such that $\tau\leq T$, $\PP^z$-a.s. We begin with the following
\begin{claim}[Change of measure and conditional expectation]
\label{claim:Change_of_measure_and_cond_exp}
For all $\cF_{\tau}$-measurable and bounded random variables, $Y,$ we have that
\begin{equation}
\label{eq:Change_of_measure_rv}
\EE_{\widehat \PP^z}\left[Y\right] = \EE_{\PP^z}\left[\frac{1}{\widehat M(\tau)}Y\right].
\end{equation}
\end{claim}

\begin{proof}[Proof of Claim \ref{claim:Change_of_measure_and_cond_exp}]
We approximate $\tau$ by a sequence of discrete stopping times, as in \cite[Problem 1.2.24]{KaratzasShreve1991}. That is, we consider the sequence of discrete stopping times $\{\tau_k\}_{k\geq 0}$, defined by $\tau_k=i 2^{-k}$ on $\{\tau\in [(i-1)2^{-k},i2^{-k}\}$, for all $i\geq 1$. By decomposing $Y=\sum_{i=1}^{\infty} Y\mathbf{1}_{\{\tau_k=i2^{-k}\}}$ and applying \cite[Lemma 3.5.3]{KaratzasShreve1991}, we have that
\begin{equation}
\label{eq:Decomposition}
\EE_{\widehat \PP^z}\left[Y\right] = \EE_{\PP^z}\left[\sum_{i=1}^{\infty}\frac{1}{\widehat M(i2^{-k})}Y\mathbf{1}_{\{\tau_k=i2^{-k}\}}\right]
= \EE_{\PP^z}\left[\frac{1}{\widehat M(\tau_k)}Y\right],\quad\forall k \in \NN.
\end{equation}
Because $\tau_k$ converges to $\tau$, as $k\rightarrow\infty$, and the paths of the process $\{\widehat Z(t)\}_{t\geq 0}$ are continuous, we see that $\widehat M^{-1}(\tau_k)Y$ converges to $\widehat M^{-1}(\tau)Y$, as $k\rightarrow\infty$. Thus, if we prove that the sequence of random variables $\{\widehat M^{-1}(\tau_k)Y\}_{k\geq 0}$ is uniformly integrable, we can apply \cite[Theorem 16.13]{Billingsley_1986}, to conclude that
$$
\EE_{\PP^z}\left[\frac{1}{\widehat M(\tau_k)}Y\right]\rightarrow \EE_{\PP^z}\left[\frac{1}{\widehat M(\tau)}Y\right],\quad\hbox{as } k \rightarrow\infty.
$$ 
The preceding property and identity \eqref{eq:Decomposition} yield \eqref{eq:Change_of_measure_rv}. To complete the proof of the claim, it remains to show that the sequence of random variables $\{\widehat M^{-1}(\tau_k)Y\}_{k\geq 0}$ is uniformly integrable. Using the \cite[Remark following inequality (16.23)]{Billingsley_1986}, it is sufficient to prove that for some $p>0$, we have that
\begin{equation}
\label{eq:Uniform_integrability}
\sup_{k \geq 0} \EE_{\PP^z}\left[\left|\widehat M^{-1}(\tau_k)Y\right|^{1+p}\right]<\infty.
\end{equation}
Since $Y$ is a bounded random variable, without loss of generality, we may assume that $Y\equiv 1$ in the preceding inequality. From identity \eqref{eq:Definition_probability_measure_SDE_singular}, we have that
\begin{align*}
&\EE_{\PP^z}\left[\left|\widehat M^{-1}(\tau_k)Y\right|^{1+p}\right] 
= \EE_{\widehat\PP^z}\left[\left|\widehat M(\tau_k)Y\right|^{-p}\right]\\
&\qquad= \EE_{\widehat\PP^z}\left[
\hbox{exp}\left(\int_0^{\tau_k} (-p\theta(\widehat Z(s)))\cdot\, d\widehat W(s)-p^2\int_0^{\tau_k} |\theta(\widehat Z(s))|^2\, ds\right) \right.\\
&\qquad\qquad\hbox{exp}\left.\left(\left(\frac{p}{2}+p^2\right)\int_0^{\tau_k}|\theta(\widehat Z(s))|^2\, ds\right) \right].
\end{align*}
Applying H\"older's inequality, we have that
\begin{align*}
\EE_{\PP^z}\left[\left|\widehat M^{-1}(\tau_k)Y\right|^{1+p}\right] 
&\leq \EE_{\widehat\PP^z}\left[
\hbox{exp}\left(\int_0^{\tau_k} (-2p\theta(\widehat Z(s)))\cdot\, d\widehat W(s)-2p^2\int_0^{\tau_k} |\theta(\widehat Z(s))|^2\, ds\right)\right]^{1/2} \\
&\quad\EE_{\widehat\PP^z}\left[\hbox{exp}\left(\left(p+2p^2\right)\int_0^{\tau_k}|\theta(\widehat Z(s))|^2\, ds\right) \right]^{1/2}.
\end{align*}
Lemma \ref{lem:Novikov_cond} together with inequality \eqref{eq:Upper_bound_theta} gives us that the process
$$
\left\{\int_0^t (-2p\theta(\widehat Z(s)))\cdot\, d\widehat W(s)-2p^2\int_0^t |\theta(\widehat Z(s))|^2\, ds\right\}_{0\leq t\leq T+1}
$$
is a $\widehat\PP^z$-martingale, and using the fact that $\tau_k\leq T+1$, we have that
$$
\EE_{\widehat\PP^z}\left[
\hbox{exp}\left(\int_0^{\tau_k} (-2p\theta(\widehat Z(s)))\cdot\, d\widehat W(s)-2p^2\int_0^{\tau_k} |\theta(\widehat Z(s))|^2\, ds\right)\right]=1,\quad\forall k\geq 0.
$$
Thus, it follows that
\begin{align*}
\EE_{\PP^z}\left[\left|\widehat M^{-1}(\tau_k)Y\right|^{1+p}\right] 
&\leq \EE_{\widehat\PP^z}\left[\hbox{exp}\left(\left(p+2p^2\right)\int_0^T|\theta(\widehat Z(s))|^2\, ds\right) \right]^{1/2},
\end{align*}
where we used the fact that $\tau_k\leq T+1$, for all $k\geq 0$. Lemma \ref{lem:Novikov_cond} and inequality \eqref{eq:Upper_bound_theta} yield that the right-hand side of the preceding inequality is bounded. This implies that inequality \eqref{eq:Uniform_integrability} holds. This completes the proof that the sequence of random variables  $\{\widehat M^{-1}(\tau_k)Y\}_{k\geq 0}$ is uniformly integrable, and the proof of the claim.
\end{proof}

We use Claim \ref{claim:Change_of_measure_and_cond_exp} to prove that for all $t\geq 0$ and all $\cF_{\tau+t}$-measurable and bounded random variables, $Z$, we have that
\begin{equation}
\label{eq:Change_of_measure_rv_cond_exp}
\EE_{\PP^z}\left[Z|\cF_{\tau}\right] = \frac{1}{\widehat M(\tau)}\EE_{\widehat \PP^z}\left[\widehat M(\tau+t)Z\Big{|}\cF_{\tau}\right].
\end{equation}
The preceding identity gives us the analogue of \cite[Lemma 3.5.3]{KaratzasShreve1991} for general stopping times, as opposed to deterministic stopping times. To see the validity of identity \eqref{eq:Change_of_measure_rv_cond_exp}, it is sufficient to show that, for all sets $A \in \cF_{\tau}$, we have
\begin{equation}
\label{eq:Change_of_measure_rv_cond_exp_reformulation}
\int Z \mathbf{1}_A \, d\PP^z = \int \frac{1}{\widehat M(\tau)}\EE_{\widehat \PP^z}\left[\widehat M(\tau+t)Z\Big{|}\cF_{\tau}\right] \mathbf{1}_A \, d\PP^z.
\end{equation}
Applying identity \eqref{eq:Change_of_measure_rv} on the right-hand side of the preceding identity, with the choice $Y:=\EE_{\widehat \PP^z}[\widehat M(\tau+t)Z|\cF_{\tau}]$, we see that
$$
\int \frac{1}{\widehat M(\tau)}\EE_{\widehat \PP^z}\left[\widehat M(\tau+t)Z\Big{|}\cF_{\tau}\right] \mathbf{1}_A \, d\PP^z 
= \int \EE_{\widehat \PP^z}\left[\widehat M(\tau+t)Z\Big{|}\cF_{\tau}\right] \mathbf{1}_A \, d\widehat\PP^z,
$$
and using the tower property of conditional expectation on the right-hand side, it follows that
$$
\int \frac{1}{\widehat M(\tau)}\EE_{\widehat \PP^z}\left[\widehat M(\tau+t)Z\Big{|}\cF_{\tau}\right] \mathbf{1}_A \, d\PP^z 
= \int \widehat M(\tau+t)Z \mathbf{1}_A \, d\widehat\PP^z.
$$
Another application of identity \eqref{eq:Change_of_measure_rv} with $Y:=Z \mathbf{1}_A$, and $\tau$ replaced by $\tau+t$, gives us that \eqref{eq:Change_of_measure_rv_cond_exp_reformulation} holds, which implies identity \eqref{eq:Change_of_measure_rv_cond_exp}.

We now prove that the strong Markov property \eqref{eq:Strong_Markov_singular} holds. We have
\begin{align}
\PP^z\left(Z(t+\tau) \in B|\cF_{\tau}\right) & = \EE_{\PP^z}\left[\mathbf{1}_{\{Z(\tau+t)\in B\}} \big{|} \cF_{\tau}\right]\notag\\
&= \EE_{\widehat \PP^z}\left[\frac{\widehat M(\tau+t)}{\widehat M(\tau)}\mathbf{1}_{\{\widehat Z(\tau+t)\in B\}} \Big{|} \cF_{\tau}\right] 
\quad\hbox{(by \eqref{eq:Change_of_measure_rv_cond_exp})}\notag \\
\label{eq:Cond_exp_reformulation_1}
&=\EE_{\widehat \PP^z}\left[\frac{\widehat M(\tau+t)}{\widehat M(\tau)}\mathbf{1}_{\{\widehat Z(\tau+t)\in B\}} \Big{|} \widehat Z(\tau)\right],
\end{align}
where the last equality follows from the strong Markov property of the process $\{\widehat Z(t)\}_{t \geq 0}$ (Corollary \ref{cor:Strong_Markov}). We want to show that
\begin{equation}
\label{eq:Cond_exp_reformulation_2}
\EE_{\widehat \PP^z}\left[\frac{\widehat M(\tau+t)}{\widehat M(\tau)}\mathbf{1}_{\{\widehat Z(\tau+t)\in B\}} \Big{|} \widehat Z(\tau)\right] 
= \PP^z(Z(\tau+t)\in B|Z(\tau)).
\end{equation}
The preceding equality holds, if for all measurable sets $A \in \cF_{\tau}$, we have that
$$
\int \EE_{\widehat \PP^z}\left[\frac{\widehat M(\tau+t)}{\widehat M(\tau)}\mathbf{1}_{\{\widehat Z(\tau+t)\in B\}} \Big{|} \widehat Z(\tau)\right] \mathbf{1}_{\{Z(\tau)\in A\}} \, d\PP^z
=\int \mathbf{1}_{\{Z(\tau+t)\in B\}}\mathbf{1}_{\{Z(\tau)\in A\}} \, d\PP^z.
$$
From identity \eqref{eq:Change_of_measure_rv}, it follows that
\begin{align*}
&\int \EE_{\widehat \PP^z}\left[\frac{\widehat M(\tau+t)}{\widehat M(\tau)}\mathbf{1}_{\{\widehat Z(\tau+t)\in B\}} \Big{|} \widehat Z(\tau)\right] \mathbf{1}_{\{Z(\tau)\in A\}} \, d\PP^z\\
&\qquad= \int \EE_{\widehat \PP^z}\left[\frac{\widehat M(\tau+t)}{\widehat M(\tau)}\mathbf{1}_{\{\widehat Z(\tau+t)\in B\}} \Big{|} \widehat Z(\tau)\right] \mathbf{1}_{\{Z(\tau)\in A\}} \widehat M(\tau)\, d\widehat\PP^z,
\end{align*}
and, the tower property and another application of identity \eqref{eq:Change_of_measure_rv}, give us
\begin{align*}
&\int \EE_{\widehat \PP^z}\left[\frac{\widehat M(\tau+t)}{\widehat M(\tau)}\mathbf{1}_{\{\widehat Z(\tau+t)\in B\}} \Big{|} \widehat Z(\tau)\right] \mathbf{1}_{\{Z(\tau)\in A\}} \, d\PP^z\\
&\qquad= \int \widehat M(\tau+t)\mathbf{1}_{\{\widehat Z(\tau+t)\in B\}}  \mathbf{1}_{\{\widehat Z(\tau)\in A\}} \, d\widehat\PP^z\\
&\qquad= \int \mathbf{1}_{\{Z(\tau+t)\in B\}}  \mathbf{1}_{\{Z(\tau)\in A\}} \, d\PP^z.
\end{align*}
Since the preceding identity is true for all measurable sets, $A\in\cF_{\tau}$, it follows that \eqref{eq:Cond_exp_reformulation_2} holds. Identities \eqref{eq:Cond_exp_reformulation_1} and \eqref{eq:Cond_exp_reformulation_2} imply \eqref{eq:Strong_Markov_singular}, and so the process $\{Z(t)\}_{t\geq 0}$ satisfies the strong Markov property.
\end{step}
This completes the proof.
\end{proof}

\subsection{Uniqueness and the strong Markov property}
\label{subsec:Uniqueness_SDE_singular}
Using the uniqueness in law of solutions to the standard Kimura stochastic differential equation \eqref{eq:Kimura_SDE}, we can establish uniqueness in law of solutions to the Kimura stochastic differential equation with singular drift \eqref{eq:Kimura_SDE_singular}, in the class of Markov processes.

\begin{thm}[Uniqueness in law of solutions to Kimura equation with singular drift]
\label{thm:Uniqueness_SDE_singular}
Suppose that the coefficients of the Kimura stochastic differential equation with singular drift \eqref{eq:Kimura_SDE_singular} satisfy Assumption \ref{assump:Existence_SDE_singular}. Then, for all $z\in\bar S_{n,m}$, there is a unique weak solution to the stochastic differential equation \eqref{eq:Kimura_SDE_singular} that satisfies the Markov property, with initial condition $Z(0)=z$. 
\end{thm}

We remark that Theorem \ref{thm:Uniqueness_SDE_singular} establishes uniqueness of solutions only in the class of Markov processes. The reason for this restriction is due to our method of the proof which consists in applying Girsanov's Theorem to remove the singular drift in equation \eqref{eq:Kimura_SDE_singular} and reduce our problem to a standard Kimura equation \eqref{eq:Kimura_SDE}, for which we know that uniqueness in law of solutions holds by Proposition \ref{prop:Uniqueness_SDE}. In applying Girsanov's Theorem, we need to establish the fact that the process $\{M(t)\}_{t \geq 0}$ defined in \eqref{eq:RN_derivative} is a martingale. As we can see from the proofs of Lemmas \ref{lem:Novikov_cond_singular} and \ref{lem:Khasminskii_singular}, this requires us to assume that the solution to the singular Kimura equation \eqref{eq:Kimura_SDE_singular} satisfies the Markov property. This is the reason why our method of the proof yields uniqueness of solutions only in the class of Markov processes.

We begin with the analogue of Lemma \ref{lem:Novikov_cond} for Kimura diffusions with singular drift.
\begin{lem}[Novikov's condition for Kimura diffusions with singular drift]
\label{lem:Novikov_cond_singular}
Suppose that the coefficients of the Kimura stochastic differential equation with singular drift \eqref{eq:Kimura_SDE_singular} satisfy Assumption \ref{assump:Existence_SDE_singular}. Let $q\in (0,q_0)$, where the positive constant $q_0$ is given by \eqref{eq:Choice_q_0}. Then for all $\Lambda>0$ and $T>0$, we have
\begin{equation}
\label{eq:Novikov_cond_singular}
\sup_{z\in\bar S_{n,m}}\EE_{\PP^{z}}\left[\hbox{exp}\left(\Lambda\int_0^T  \sum_{i=1}^n |X_i(t)|^{-2q}\, dt\right)\right]<\infty,
\end{equation}
where $\{Z(t)=(X(t),Y(t))\}_{t\geq 0}$ is a solution to the singular Kimura stochastic differential equation \eqref{eq:Kimura_SDE_singular} that satisfies the Markov property, with initial condition $Z(0)=z$.
\end{lem}

We prove Lemma \ref{lem:Novikov_cond_singular} with the aid of the analogue of Lemma \ref{lem:Khasminskii} for the Kimura stochastic differential equation with singular drift.
\begin{lem}[Verification of the hypotheses of Khas'minskii's Lemma for singular Kimura diffusions]
\label{lem:Khasminskii_singular}
Suppose that the coefficients of the Kimura stochastic differential equation with singular drift \eqref{eq:Kimura_SDE_singular} satisfy Assumption \ref{assump:Existence_SDE_singular}. Let $q\in (0,q_0)$, where the positive constant $q_0$ is given by \eqref{eq:Choice_q_0}. Then for all positive constants, $\delta\in (0,1)$ and $\Lambda$, there is a positive constant, $T=T(b_0,\delta,K_0,K,\Lambda,m,n,q)$, such that
\begin{equation}
\label{eq:Khasminskii_singular}
\sup_{z\in\bar S_{n,m}}\EE_{\PP^{z}}\left[\Lambda\int_0^T  \sum_{i=1}^n |X_i(t)|^{-2q}\, dt\right]<\delta,
\end{equation}
where $\{Z(t)=(X(t),Y(t))\}_{t\geq 0}$ is a solution to the singular Kimura stochastic differential equation \eqref{eq:Kimura_SDE_singular}, with initial condition $Z(0)=z$.
\end{lem}

\begin{proof}
The proof of Lemma \ref{lem:Khasminskii_singular} is similar to that of Lemma \ref{lem:Khasminskii}, but we have to pay closer attention to the singular component of the drift coefficient in the stochastic differential equation \eqref{eq:Kimura_SDE_singular}. We let the positive constants $\rho$, $r$ and the cut-off function $\varphi$ be as in the proof of Lemma \ref{eq:Khasminskii}. Without loss of generality, we may assume that $\Lambda=1$. We consider the auxiliary function
\begin{equation*}
\psi(x)=\sum_{i=1}^n x_i^{1-2q}\varphi(x_i) ,\quad\forall\, x\in\RR^n_+.
\end{equation*}
For all $\eps\in (0,1)$, we recall that we denote $X^{\eps}_i(t)=X_i(t)+\eps$ and $x_i^{\eps}=x_i+\eps$. Applying It\^o's rule to the process $\{\psi(X^{\eps}(t))\}_{t\geq 0}$, we obtain, for all $T>0$ and $\eps>0$,
\begin{equation}
\label{eq:Ito_rule_power_function_prim}
\begin{aligned}
&\psi(X^{\eps}(T)) 
= \psi(x^{\eps}) + \sum_{i=1}^n \int_0^T b_i(Z(t))\varphi'(X^{\eps}_i(t))(X^{\eps}_i(t))^{1-2q}\, dt\\
&\quad + \sum_{i=1}^n\int_0^T \frac{|\sigma_i(Z(t))|^2}{2}X_i(t) (X^{\eps}_i(t))^{-2q}\left(X^{\eps}_i(t)\varphi''(X^{\eps}_i(t))  
+2(1-2q)\varphi'(X^{\eps}_i(t))\right)\, dt\\
&\quad+(1-2q) \sum_{i=1}^n\int_0^T \varphi(X^{\eps}_i(t))(X^{\eps}_i(t))^{-2q} \left(b_i(Z(t))- q|\sigma_i(Z(t))|^2\frac{X_i(t)}{X^{\eps}_i(t)}\right)\, dt\\
&\quad+(1-2q)\sum_{i=1}^n \int_0^T \varphi(X^{\eps}_i(t))(X^{\eps}_i(t))^{-2q}\sqrt{X_i(t)} \sum_{j=1}^nf_{ij}(Z(t))h_{ij}(X_j(t))\, dt\\
&\quad+\sum_{i=1}^n \int_0^T (X^{\eps}_i(t))^{1-2q}\varphi'(X^{\eps}_i(t)) \sqrt{X_i(t)}\sum_{j=1}^n f_{ij}(Z(t))h_{ij}(X_j(t))\, dt\\
&\quad+ \sum_{i=1}^n\int_0^T (X^{\eps}_i(t))^{-2q}\left(X^{\eps}_i(t)\varphi'(X^{\eps}_i(t)) +(1-2q)\varphi(X^{\eps}_i(t))\right)\sqrt{X_i(t)}\sigma_i(Z(t))\cdot\, d W(t),
\end{aligned}
\end{equation}
where we recall that $\sigma_i(z)$ denotes the $i$-th row of the matrix $(\sigma(z))$. By Remark \ref{rmk:Boundedness_coeff_SDE}, the coefficient functions $b(z)$ and $(\sigma(z))$ are bounded, and using the properties of the cut-off function $\varphi$, there is a positive constant, $C=C(K,m,n)$, such that
\begin{equation}
\label{eq:Ito_rule_power_function_prim_2}
\begin{aligned}
&\psi(x^{\eps}) + \sum_{i=1}^n \int_0^T b_i(Z(t))\varphi'(X^{\eps}_i(t))(X^{\eps}_i(t))^{1-2q}\, dt\\
&+ \sum_{i=1}^n\int_0^T \frac{|\sigma_i(Z(t))|^2}{2}X_i(t) (X^{\eps}_i(t))^{-2q}\left(X^{\eps}_i(t)\varphi''(X^{\eps}_i(t))  
+2(1-2q)\varphi'(X^{\eps}_i(t))\right)\, dt\\
&\leq nr^{1-2q} + Cr^{-2q}T,\quad\forall\, \eps>0.
\end{aligned}
\end{equation}
Inequality \eqref{eq:Lower_bound}, together with \eqref{eq:Ito_rule_power_function_prim}, \eqref{eq:Ito_rule_power_function_prim_2} and the fact that the cut-off function $\varphi$ has support in $[0,r]$, yields
\begin{equation}
\label{eq:Inequality_1}
\begin{aligned}
&C_0\sum_{i=1}^n\int_0^T \left|X^{\eps}_i(t)\right|^{-q}\mathbf{1}_{\{X_i(t)\in [0,r/2]\}}\,dt
\leq  2nr^{1-2q} + Cr^{-2q}T \\
&\qquad+r^{-2q}\sum_{i=1}^n\int_0^{T} \sqrt{X_i(t)} \mathbf{1}_{\{X_i(t)\in[0,r]\}}\left|\sum_{j=1}^n f_{ij}(Z(t))h_{ij}(X_j(t))\right|\, dt\\
&\qquad- \sum_{i=1}^n\int_0^T \sqrt{X_i(t)}(X^{\eps}_i(t))^{-2q}\left(X^{\eps}_i(t)\varphi'(X^{\eps}_i(t)) +(1-2q)\varphi(X^{\eps}_i(t))\right)\sigma_i(Z(t))\cdot\, d W(t),
\end{aligned}
\end{equation}
for all $\eps>0$, where we denote for brevity
$$
C_0:=(1-2q)\left(\frac{b_0}{1+\rho}-q(n+m)K^2\right).
$$
Note that the $dt$-integral term on the right-hand side of inequality \eqref{eq:Inequality_1} is finite, from our assumption that the process $\{Z(t)\}_{t\geq 0}$ is a weak solution to equation \eqref{eq:Kimura_SDE_singular}, which implies that
$$
\int_0^{T} \sqrt{X_i(t)} \left|\sum_{j=1}^n f_{ij}(Z(t))h_{ij}(X_j(t))\right|\, dt<\infty,\quad\PP^z\hbox{-a.s.}
$$
Moreover using the fact that $q< 1/4$, from \eqref{eq:Choice_q_0}, we see that $\sqrt{x_i}(x^{\eps}_i)^{-2q}$ is bounded as $\eps\downarrow 0$. From Remark \ref{rmk:Boundedness_coeff_SDE}, it follows that the matrix coefficient $(\sigma(z))$ is bounded, and so, the $dW(t)$-integral on the right-hand side of inequality \eqref{eq:Inequality_1} converges, when we take limit as $\eps\downarrow 0$. Inequality \eqref{eq:Inequality_1} becomes, as $\eps\downarrow 0$,
\begin{align*}
&C_0\sum_{i=1}^n\int_0^T \left|X_i(t)\right|^{-2q}\mathbf{1}_{\{X_i(t)\in [0,r/2]\}}\,dt  
\leq  2nr^{1-2q}+ Cr^{-2q} T \\
&\qquad+r^{-2q}\sum_{i=1}^n\int_0^{T} \sqrt{X_i(t)} \mathbf{1}_{\{X_i(t)\in[0,r]\}}\left|\sum_{j=1}^n f_{ij}(Z(t))h_{ij}(X_j(t))\right|\, dt\\
&\qquad- \sum_{i=1}^n\int_0^T \sqrt{X_i(t)}(X_i(t))^{-2q}\left(X_i(t)\varphi'(X_i(t)) +(1-2q)\varphi(X_i(t))\right)\sigma_i(Z(t))\cdot\, d W(t),
\end{align*}
and so, the integral on the left-hand side of the preceding inequality is finite. We may now use the upper bounds \eqref{eq:Boundedness_coeff_SDE}, \eqref{eq:Boundedness_f} and \eqref{eq:Singularity_h}, to conclude that there is a positive constant, $C=C(K_0,K,m,n)$, such that
\begin{align*}
&\sum_{i=1}^n\int_0^T \left|X_i(t)\right|^{-2q}\mathbf{1}_{\{X_i(t)\in [0,r/2]\}}\,dt  
\leq  \frac{2nr^{1-2q}+ Cr^{-2q} T}{C_0}\\
&\qquad +\frac{Cr^{1/2-2q}}{C_0}\sum_{i=1}^n\int_0^{T} |X_i(t)|^{-q}\mathbf{1}_{\{X_i(t)\in [0,r/2]\}}\, dt\\
&\qquad- \frac{1}{C_0}\sum_{i=1}^n\int_0^T \sqrt{X_i(t)}(X_i(t))^{-2q}\left(X_i(t)\varphi'(X_i(t)) +(1-2q)\varphi(X_i(t))\right)\sigma_i(Z(t))\cdot\, d W(t),
\end{align*}
Because we choose $q< 1/4$, from \eqref{eq:Choice_q_0}, we may choose a positive constant, $r_1=r_1(C_0,C,q)$, small enough so that $Cr_1^{1/2-2q}/C_0\leq 1/2$. The preceding inequality gives us that
\begin{align*}
&\sum_{i=1}^n\int_0^T \left|X_i(t)\right|^{-2q}\mathbf{1}_{\{X_i(t)\in [0,r/2]\}}\,dt  
\leq  \frac{4nr^{1-2q}+ 2Cr^{-2q} T}{C_0}\\
&\qquad- \frac{2}{C_0}\sum_{i=1}^n\int_0^T \sqrt{X_i(t)}(X_i(t))^{-2q}\left(X_i(t)\varphi'(X_i(t)) +(1-2q)\varphi(X_i(t))\right)\sigma_i(Z(t))\cdot\, d W(t),
\end{align*}
for all $r\in (0,r_1)$. Because the $dW(t)$-term in the preceding equality defines a martingale, we may take expectation in the preceding inequality to obtain, for all $T>0$ and $r\in (0,r_1)$,
\begin{align*}
\EE_{\PP^z}\left[\int_0^T \sum_{i=1}^n\left|X_i(t)\right|^{-2q}\mathbf{1}_{\{X_i(t)\in [0,r/2]\}}\,dt\right]
&\leq \frac{2nr^{1-2q}+ Cr^{-2q} T}{C_0}.
\end{align*}
Removing the indicator function in the preceding inequality, we obtain 
\begin{align*}
\EE_{\PP^z}\left[\int_0^T \sum_{i=1}^n\left|X_i(t)\right|^{-2q}\,dt\right]
&\leq \frac{2nr^{1-2q}+ Cr^{-2q} T}{C_0}+nr^{-2q}T.
\end{align*}
For all $\delta\in (0,1)$, we may now choose $r=r(b_0,\delta,K_0,K,m,n,q)$ and $T=T(b_0,\delta,K_0,K,m,n,q)$ small enough so that inequality \eqref{eq:Khasminskii_singular} holds. This completes the proof.
\end{proof}

We can now give the
\begin{proof}[Proof of Lemma \ref{lem:Novikov_cond_singular}]
The proof of Lemma \ref{lem:Novikov_cond_singular} is identical to that of Lemma \ref{lem:Novikov_cond}, only in place of Lemma \ref{lem:Khasminskii} we use Lemma \ref{lem:Khasminskii_singular}, and so, we omit the detailed proof.
\end{proof}

To prove Theorem \ref{thm:Uniqueness_SDE_singular}, in addition to Lemma \ref{lem:Novikov_cond_singular}, we need the following result which proves uniqueness in law of the joint probability distributions of any weak solution $\{(\widehat Z(t),\widehat W(t))\}_{t\geq 0}$, to the standard Kimura stochastic differential equation, \eqref{eq:Kimura_SDE}.

\begin{lem}[Uniqueness of the joint law of weak solutions $(\widehat Z,\widehat W)$ to the standard Kimura equation]
\label{lem:Uniqueness_joint_law}
Suppose that Assumption \ref{assump:Uniqueness_SDE} holds. Let $z\in\bar S_{n,m}$, and let $(\widehat Z^i,\widehat W^i)$, $(\Omega^i,\{\cF^i_t\}_{t\geq 0}, \cF^i,\widehat \PP^z_i)$, for $i=1,2$, be two weak solutions to the standard Kimura equation \eqref{eq:Kimura_SDE}, with initial condition $\widehat Z^1(0)=\widehat Z^2(0)=z$. Then the joint probability laws of the processes $(\widehat Z^i,\widehat W^i)$, for $i=1,2$, agree.
\end{lem}

\begin{proof}
From Proposition \ref{prop:Uniqueness_SDE}, it follows that the probability laws of the processes $\{\widehat Z^i(t)\}_{t\geq 0}$, for $i=1,2$, agree. For $i=1,2$, we consider the $(n+m)$-dimensional processes defined by
\begin{equation}
\label{eq:Definition_process_N}
\begin{aligned}
\widehat N^i_j(t) &:= \widehat X^i_j(t) - \widehat X^i_j(0)-\int_0^t \widehat b_j(\widehat Z^i(s))\, ds,\quad\forall\, j=1,\ldots,n,\\
\widehat N^i_{n+l}(t) &:= \widehat Y^i_l(t)-\widehat Y^i_l(0)-\int_0^t e_l(\widehat Z^i(s))\, ds,\quad\forall\, l=1,\ldots,m.
\end{aligned}
\end{equation}
Our goal is to prove that the following identity holds
\begin{equation}
\label{eq:BM_in_terms_of_Z}
\widehat W^i(t) = \int_0^t \varsigma^{-1}(\widehat Z^i(s))\mathbf{1}_{\{\widehat Z^i(s)\in S_{n,m}\}}  d\widehat N^i(s),\quad\forall\, t\geq 0,
\end{equation}
$\widehat\PP^z_i$-a.s, for $i=1,2$. Notice that on the right-hand side of the preceding identity we used the invertibility of the matrix function $(\varsigma(z))$, for all $z\in S_{n,m}$. This follows from identity \eqref{eq:Diffusion_matrix} and the fact that the matrix $(D(z))$ is invertible, by Remark \ref{rmk:Invertibility_sigma}. Property \eqref{eq:BM_in_terms_of_Z} and \eqref{eq:Definition_process_N}, together with the fact that the probability laws of the processes $\{\widehat Z^i(t)\}_{t\geq 0}$, for $i=1,2$, agree, imply that the joint probability laws of the processes $(\widehat Z^i,\widehat W^i)$, for $i=1,2$, also agree.

We now proceed to the proof of identity \eqref{eq:BM_in_terms_of_Z}. From definition \eqref{eq:varsigma} of the matrix function $(\varsigma(z))$, the choice of the processes $\{\widehat N^i(t)\}_{t\geq 0}$, for $i=1,2$, and from equation \eqref{eq:Kimura_SDE}, we see that identity \eqref{eq:BM_in_terms_of_Z} is equivalent to
$$
\widehat W^i(t) = \int_0^t \mathbf{1}_{\{\widehat Z^i(s)\in S_{n,m}\}}  d\widehat W^i(s)\quad \widehat\PP^z_i\hbox{-a.s.},\quad\forall\, t\geq 0,\quad i=1,2.
$$
Thus identity \eqref{eq:BM_in_terms_of_Z} holds if and only if
\begin{equation*}
\int_0^t \mathbf{1}_{\{\widehat Z^i(s)\in \partial S_{n,m}\}}  d\widehat W^i(s) =0 \quad \widehat\PP^z_i\hbox{-a.s.},\quad\forall\, t\geq 0,\quad i=1,2.
\end{equation*}
The preceding equality is equivalent to proving that
$$
\EE_{\widehat \PP^z_i}\left[\int_0^t \mathbf{1}_{\{\widehat Z^i(s)\in \partial S_{n,m}\}}  \, ds\right] =0,\quad\forall\, t\geq 0,\quad i=1,2,
$$
but this clearly holds from the fact that the quantity defined in \eqref{eq:Novikov_cond} is finite. This completes the proof.
\end{proof}

We can now give the proof of Theorem \ref{thm:Uniqueness_SDE_singular} with the aid of Lemmas \ref{lem:Novikov_cond_singular} and \ref{lem:Uniqueness_joint_law}.

\begin{proof}[Proof of Theorem \ref{thm:Uniqueness_SDE_singular}]
Let $(Z^i,W^i)$, $(\Omega^i, \{\cF^i_t\}_{0\leq t\leq T},\cF^i,\PP_i^z)$, be two weak solutions to the Kimura stochastic differential equation with singular drift \eqref{eq:Kimura_SDE_singular}, satisfying the initial condition $Z^i(0)=z$, for $i=1,2$, where we assume that $z\in\bar S_{n,m}$. Assume that the two weak solutions satisfy the Markov property. Our goal is to show that the laws of the processes $\{Z^i(t)\}_{t\in [0,T]}$, are the same, for all $T>0$, for $i=1,2$. Let $\theta:S_{n,m}\rightarrow\RR^{n+m}$ be the vector field defined in \eqref{eq:Definition_theta}, and recall that the function $\theta(z)$ satisfies inequality \eqref{eq:Upper_bound_theta}. Lemma \ref{lem:Novikov_cond_singular} together with inequality \eqref{eq:Upper_bound_theta} shows that condition \eqref{eq:Novikov_cond_singular} holds, and so, \cite[Corollary 3.5.13]{KaratzasShreve1991} yields that the processes $\{M^i(t)\}_{0\leq t\leq T}$ defined by
\begin{equation}
\label{eq:RN_derivative}
M^i(t):=\hbox{exp}\left(-\int_0^t \theta(Z^i(s))\cdot\, dW^i(s)- \frac{1}{2}\int_0^t |\theta(Z^i(s))|^2\, ds\right),\quad \forall t \in [0,T],
\end{equation}
are $\PP^{z}_i$-martingale, for $i=1,2$. We can apply Girsanov's Theorem \cite[Theorem 3.5.1]{KaratzasShreve1991} to construct new probability measures, $\widehat\PP^{z}_i$, by letting
\begin{equation}
\label{eq:Definition_probability_measure_SDE}
\frac{d\widehat\PP^{z}_i}{d\PP^{z}_i} = M^i(T),\quad i=1,2.
\end{equation}
Then the process 
\begin{equation}
\label{eq:Widehat_BM}
\widehat W^i(t) := W^i(t)+\int_0^t \theta(Z^i(s))\, dt , \quad\forall t\in [0,T],
\end{equation}
is a Brownian motion with respect to the probability measure $\widehat\PP^z_i$, for $i=1,2$. Using definition \eqref{eq:Definition_theta} of the function $\theta(z)$, we see that by letting $\widehat Z^i(t):=  Z^i(t)$, for all $t\in [0,T]$, we obtain that the processes $\{\widehat Z^i(t),\widehat W^i(t)\}_{0\leq t\leq T}$, $(\Omega,\{\cF_t\}_{0\leq t\leq T}, \cF, \widehat\PP^z_i)$ are weak solutions to the standard Kimura stochastic differential equation \eqref{eq:Kimura_SDE}, with initial condition $\widehat Z^i(0)=z$, for $i=1,2$. From Lemma \ref{lem:Uniqueness_joint_law}, it follows that the joint law of the processes $\{\widehat Z^i(t), \widehat  W^i(t)\}_{0\leq t\leq T}$, for $i=1,2$, agree. From definitions \eqref{eq:RN_derivative} and \eqref{eq:Widehat_BM}, we have that
$$
M^i(t):=\hbox{exp}\left(-\int_0^t \theta(\widehat Z^i(s))\cdot\, d\widehat  W^i(s) + \frac{1}{2}\int_0^t |\theta(\widehat Z^i(s))|^2\, ds\right),\quad i=1,2,
$$ 
and so, the laws of the processes $\{M^1(t)\}_{0\leq t\leq T}$ and $\{M^2(t)\}_{0\leq t\leq T}$ also agree. Thus, it follows from \eqref{eq:Definition_probability_measure_SDE}, that the probability laws of the processes $\{Z^i(t)\}_{t\in [0,T]}$ are the same, for all $T>0$, for $i=1,2$. This concludes the proof.
\end{proof}

%
%

\bibliography{mfpde}
\bibliographystyle{amsplain}

\end{document}